\documentclass[reqno]{amsart}
\usepackage{amssymb,amsmath,amsthm,amstext,amsfonts}
\usepackage{amsmath,amstext,amsthm,amsfonts}
\usepackage[dvips]{graphicx}
\usepackage{epic,eepic}
\usepackage[dvips]{graphicx}
\usepackage{psfrag}
\usepackage{color}

\makeatletter \@addtoreset{equation}{section} \makeatother

\renewcommand\thetable{\thesection.\@arabic\c@table}

\theoremstyle{plain}
\newtheorem{maintheorem}{Theorem}

\newtheorem{maincorollary}{Corollary}
\newtheorem{theorem}{Theorem }[section]
\newtheorem{proposition}[theorem]{Proposition}
\newtheorem{lemma}[theorem]{Lemma}
\newtheorem{corollary}[theorem]{Corollary}

\theoremstyle{definition} \theoremstyle{remark}
\newtheorem{remark}[theorem]{Remark}
\newtheorem{example}[theorem]{Example}

\newcommand{\field}[1]{\mathbb{#1}}
\newcommand{\real}{\field{R}}

\renewcommand{\natural}{\field{N}}

\newcommand{\torus}{\field{T}}

\newcommand{\al} {\alpha}       
        
\newcommand{\ga} {\gamma}    
\newcommand{\de} {\delta}       \newcommand{\De}{\Delta}

\newcommand{\vep}{\varepsilon}

\newcommand{\la} {\lambda}      \newcommand{\La}{\Lambda}

\newcommand{\si} {\sigma}

\newcommand{\N}{\mathbb{N}}
\newcommand{\R}{\mathbb{R}}
\newcommand{\supp}{\operatorname{supp}}
\newcommand{\diam}{\operatorname{diam}}

\newcommand{\topp}{\operatorname{top}}

\renewcommand{\field}[1]{\mathbb{#1}}

\newcommand{\re}{\field{R}}

\renewcommand{\natural}{\field{N}}
\newcommand{\vr}{\varphi}

\newcommand{\Ptop}{P_{\topp}}

\newcommand{\un}{\underbar}

\newcommand{\cP}{\mathcal{P}}
\newcommand{\cL}{\mathcal{L}}

\newcommand{\cM}{\mathcal{M}}
\newcommand{\cB}{\mathcal{B}}

\newcommand{\cF}{\mathcal{F}}

\newcommand{\cW}{\mathcal{W}}
\newcommand{\cU}{\mathcal{U}}

\newcommand{\cA}{\mathcal{A}}

\begin{document}

\title{Equilibrium states for non-uniformly expanding maps: decay of correlations and strong stability}
\author{ A. Castro and P. Varandas}

\address{Armando Castro, Departamento de Matem\'atica, Universidade Federal da Bahia\\
Av. Ademar de Barros s/n, 40170-110 Salvador, Brazil.}
\email{armando@impa.br}

\address{Paulo Varandas, Departamento de Matem\'atica, Universidade Federal da Bahia\\
Av. Ademar de Barros s/n, 40170-110 Salvador, Brazil.}
\email{paulo.varandas@ufba.br}

\date{\today}

\maketitle

\begin{abstract}
We study the rate of decay of correlations for equilibrium states associated
to a robust class of non-uniformly expanding maps where no Markov assumption is required.
We show that the Ruelle-Perron-Frobenius operator acting on the space of H\"older continuous 
observables has a spectral gap and deduce the exponential decay of correlations and the central limit theorem.
In particular, we obtain an alternative proof for the existence and uniqueness of the equilibrium states
and we prove that the topological pressure varies continuously.
Finally, we use the spectral properties of the transfer operators in space of differentiable observables
to obtain strong stability results under deterministic and random perturbations.  
\end{abstract}


\section{Introduction}

The thermodynamical formalism was brought from statistical mechanics to dynamical systems 
by the pioneering works of Sinai, Ruelle and Bowen \cite{Si72,Bo75,BR75} in the mid seventies. 
Indeed, the correspondance between one-dimensional lattices and uniformly hyperbolic maps, via 
Markov partitions, allowed to translate and introduce several notions of Gibbs measures and 
equilibrium states in the realm of dynamical systems.
Nevertheless, although uniformly hyperbolic dynamics arise in physical systems
(see e.g. \cite{HM03}) they do not include some relevant classes of
systems including the Manneville-Pomeau transformation (phenomena of
intermittency), H\'enon maps and billiards with convex scatterers.
We note that all the previous systems present some non-uniformly
hyperbolic behavior and its relevant measure satisfies some weak
Gibbs property. 
Moreover, an extension of the thermodynamical formalism beyond the scope of uniform 
hyperbolicity reveals fundamental difficulties. Even in the non-uniformly hyperbolic context, 
where there are no zero Lyapunov exponents and there exists a non-uniform geometric theory of invariant
manifolds, the absence of finite generating Markov partitions constitutes an obstruction to use the
same strategy pushed forward before. Nevertheless, more recently there have been established
many evidences that non-uniformly hyperbolic dynamical systems admit
countable and generating Markov partitions. This is now paralel to the development
of a thermodynamical formalism of gases with infinitely many states, a hard subject
not yet completely understood. We refer the reader to~\cite{BS03,Pin08,PV10} for recent progress
in this direction.

So, despite the effort of many authors, a general picture is still far from complete. 
Some of the recent contributions concerning the existence and uniqueness of equilibrium states
in a context of non-uniform hyperbolicity include 
\cite{BrK98, BM02, BS03, Yu03,OV08,SV09,BF09,SV09,VV10,PR10}.
Many of these papers deal with dynamical systems with neutral periodic points, unimodal maps,
perturbations of hyperbolic transformations and shifts with countable many symbols, some of the 
relevant sources of examples of non-hyperbolic systems.
However, a deep study on the statistical properties of the equilibrium states, as the mixing properties,
limit theorems, strong stability under deterministic and random perturbations or regularity of the
topological pressure is usually obtained as a consequence of the spectral properties of the Ruelle-
Perron-Frobenius operator. This functional analytic approach has gained special interest in the last 
few years and produced new and interesting results even in the uniformly hyperbolic setting 
(see e.g. \cite{BKL01,GL06,BT07}). Just for completeness let us mention that, since the (semi)conjugacy 
between uniformly hyperbolic dynamical systems and the symbolic dynamics is only H\"older continuous, 
the strategy developed in the seventies did not  allowed to understand the statistical properties in the space of
smooth observables. Important and recent extensions of this functional analytic approach to 
the setting of non-uniform hyperbolicity include e.g. the works \cite{LSV98, Cas02,Cas04,DL08,BG09,BG10,Ru10}.

In this article we study the strong statistical properties of some equilibrium states 
built in \cite{VV10} for a large class of non-uniformly expanding local homeomorphisms that may not admit 
a Markov partition. Using a characterization of equilibrium states as weak Gibbs measures absolutely continuous
with respect to conformal reference measures, the authors proved roughly that every local homeomorphism with coexistence of expanding and contraction exhibit a form of average expansion. This enables to
use Birkhoff's method of projective cones aplied to the Ruelle-Perron-Frobenius operator acting on suitable
Banach spaces to obtain the existence of a unique equilibrium state for any H\"older continuous potential with low variation and that it satisfies strong statistical properties.
Natural examples are obtained by bifurcation of expanding homeomorphisms and subshifts of
finite type and allows intermitency phenomena.
Even in the absense of Markov partition we establish that  the Ruelle-Perron-Frobenius transfer operator 
has a spectral gap in the Banach spaces of both H\"older continuous and smooth observables. This was inspired 
and extends the work of Matheus and Arbieto~\cite{AM} that considered local diffeomorphisms under some slightly
different assumptions but where the existence of a finite Markov partition played an important role.
In consequence, we get an alternative proof for the existence and uniqueness of equilibrium states in \cite{VV10}, obtain
exponential decay of correlations and prove a central limit theorem. Moreover, we prove that in this non-uniformly 
expanding setting the topological pressure varies continuously with respect to the dynamics and the potential.

At this point one could think the stability of the equilibrium states under deterministic and random perturbations
could follow directly from the spectral gap property. We refer the reader to \cite{HH01} for perturbation theory
of smooth families of quasi-compact operators. However this is not the case since the transfer operators
acting on the space of H\"older continuous potentials may not vary continuously on the 
dynamical system as illustrated in Example~\ref{ex.transfer}. Nevertheless we prove that the densities of the equilibrium states with respect to the conformal measures are H\"older continuous and vary uniformly with the
dynamics.
Strong statistical and stochastic stability results hold in the space of differentiable observables and are
proved after carefull analysis of the action of the transfer operators in those functional spaces.
We obtain a spectral stability under random perturbations. Namely,  the 
spectral components of the Ruelle-Perron-Frobenius operator associated to general random perturbations 
of the transformation and the potential varies continuously and converges to the spectral components of 
Ruelle-Perron-Frobenius of the the unperturbed dynamical system outside of a disk containing zero in the
spectrum. 

Finally, let us also mention that the program to understand to statistical and stochastic properties of the equilibrium
states for this class of multidimensional non-uniformly expanding transformations is under way. Some of 
the very interesting remaining questions are to understand  if one can obtain further regularity of the topological
pressure and the density of the equilibrium states with respect to conformal measures along
parametrized families of potentials (e.g. real analytic) and the study of zeta functions. Such program has been
carried out with success for uniformly hyperbolic and some partially hyperbolic and one-dimensional 
non-uniformly expanding dynamical systems. See e.g.~\cite{Rue97,Dol04,BS08,BS09} and references therein. 
Just to mention some recent developments, in a joint work with T. Bomfim \cite{BCV12}, 
we prove the differentiability of thermodynamical quantities as topological pressure, invariant densities, conformal measures and measures of maximal entropy despite the lack of continuity of the Ruelle-Perron-Frobenius operator 
with respect to the dynamics. 

This paper is organized as follows.  In Section~\ref{s.statements}, we recall some definitions and make 
the precise statements of our main results and some preliminary results are given in Section~\ref{s.preliminaries}. 
The proof of the spectral gap for the Ruelle-Perron-Frobenius operator in the space of H\"older continuous
observables, continuity of the topological pressure, uniform continuity of the densities of
equilibrium states with respect to conformal measures and exponential decay of correlations 
are given in Section~\ref{s.gapHolder}.
In Section~\ref{s.gapCr} we show that Ruelle-Perron-Frobenius operator acting on the space of smooth 
observables also admits a spectral gap and obtain the strong stability of the equilibrium states under 
deterministic and random perturbations. Finally, some examples are given in Section~\ref{s.examples}.

\section{Statement of the main results}\label{s.statements}

\subsection{Setting}

Let $M$ be compact and connected Riemmanian manifold of dimension $m$ with distance $d$. 
Let $f:M \to M$ be a \emph{local homeomorphism} and assume that there exists a continuous function
$x\mapsto L(x)$ such that, for every $x\in M$ there is a
neighborhood $U_x$ of $x$ so that $f_x : U_x \to f(U_x)$ is
invertible and
$$
d(f_x^{-1}(y),f_x^{-1}(z))
    \leq L(x) \;d(y,z), \quad \forall y,z\in f(U_x).
$$
In particular every point has the same finite number of preimages $\deg(f)$ which coincides with the 
degree of $f$.
For all our results we assume that $f$ and $\phi$ satisfy conditions
(H1), (H2), and (P) stated below. Assume there
exist constants $\si>1$ and $L\ge 1$, and an open region $\cA\subset M$
such that \vspace{.1cm}
\begin{itemize}
\item[(H1)] $L(x)\leq L$ for every $x \in \cA$ and
$L(x)< \sigma^{-1}$ for all $x\notin \cA$, and $L$ is
close to $1$: the precise condition is given in \eqref{eq.
relation expansion} and \eqref{eq. relation potential}.
\item[(H2)] There exists a finite covering $\cU$ of $M$ by open domains of injectivity for $f$ 
such that $\cA$ can be covered by $q<\deg(f)$. \vspace{.1cm}
\end{itemize}
The first condition means that we allow expanding and contracting
behavior to coexist in $M$: $f$ is uniformly expanding outside $\cA$
and not too contracting inside $\cA$. In the case that  $\cA$ is empty then
$f$ is uniformly expanding. The second one requires
that every point has at least one preimage in the expanding region. 
An observable $g: M\to \mathbb R$ is $\al$-H\"older continuous if the H\"older constant  
$$
|g|_\al
	=\sup_{x\neq y} \frac{|g(x)-g(y)|}{d(x,y)^\al}
$$
is finite. As usual, we endow the space $C^\al(M, \mathbb R)$ of H\"older continuous observables with
the norm $\|\cdot\|_\al=\|\cdot \|_0+|\cdot|_\al$.
We assume that the potential $\phi:M \to \mathbb R$ is H\"older
continuous and that 
\begin{itemize}
\item[(P)] $\sup\phi-\inf\phi<\vep_\phi$ \quad  and \quad $|e^\phi|_{\alpha}  <\vep_\phi \;e^{\inf \phi}$
\end{itemize}
for some $\vep_\phi>0$ satisfying equation~\eqref{eq.vep}, depending on the constants $L$, $\si$, $q$
and $\deg(f)$. 
The previous is an open condition on the potential, relative to the
H\"older norm, and it is satisfied e.g. by constant functions.
In particular we consider measures of maximal entropy.
The second condition above means that $\exp(\phi)$ is contained in a small cone of H\"older 
continuous as discussed after Theorem~\ref{t.cone.invariance}.

\subsection{Existence and uniqueness of equilibrium states}\label{existence eq.
states}

Let us first recall some necessary definitions. 
Given a continuous map $f:M\to M$ and a potential $\phi:M \to \mathbb R$,
the variational principle for the pressure asserts that
\begin{equation*}
\label{variational principle} \Ptop(f,\phi)=\sup \Big\{
h_\mu(f)+\int \phi \;d\mu : \mu \;\text{is}\; f\text{-invariant}
\Big\}
\end{equation*}
where $\Ptop(f,\phi)$ denotes the topological pressure of $f$ with
respect to $\phi$ and $h_\mu(f)$ denotes the metric entropy. An
\textit{equilibrium state} for $f$ with respect to $\phi$ is an
invariant measure that attains the supremum in the right hand side
above.

The equilibrium states constructed in \cite{VV10} are absolutely continuous
with respect to an expanding, conformal and non-lacunary Gibbs measure $\nu$. Let us
recall these definitions and the notions involved.
A probability measure $\nu$, not necessarily invariant, is
\emph{conformal} if there exists a function $\psi:M\to \R$ such
that $\nu(f(A))=\int_A e^{-\psi} d\nu$ 
for every measurable set $A$ such that $f \mid A$ is injective.
Let $S_n \phi=\sum_{j=0}^{n-1} \phi \circ f^j$ denote the $n$th
Birkhoff sum of a function $\phi$. 
The \emph{basin of attraction} of an $f$-invariant, ergodic probability
measure $\mu$ is the set $B(\mu)$ of points $x \in M$ such that the probability measures
$\frac{1}{n} \sum_{j=0}^{n-1} \de_{f^j(x)}$ converges weakly to $\mu$ when $n\to\infty$.
We build over the following theorem which is a direct consequence of the results in \cite{VV10}.

\begin{theorem}\label{thm.VV}
Let $f:M \to M$ be a 
local homeomorphism with Lipschitz continuous inverse satisfying (H1), (H2) and $\phi:M \to \R$ a 
H\"older continuous potential such that $\sup\phi-\inf\phi<\log\deg(f)-\log q$.
Then, there exists a finite number of ergodic equilibrium states $\mu_1, \mu_2, \dots, \mu_k$
for $f$ with respect to $\phi$, and they are absolutely continuous with respect to some conformal 
expanding measure $\nu$. Moreover, the union of the basins of attraction  $B(\mu_i)$ contain 
$\nu$-almost every point.
 \end{theorem}

Observe that despite the characterization that equilibrium states are absolutely continuous
invariant measures no information was known e.g. on the continuity of the topological pressure and density  functions.  Here we shall adress these questions, the uniqueness of the equilibrium states and also the strong
 stability of the equilibrium states. Since our assumption (P) implies that the potential $\phi$ has small variation 
 then it fits in the assumption of the previous theorem. We will build over the aforementioned result with a 
 completely  different functional analytic approach.

\subsection{Statement of the main results}

In this section we recall some necessary definitions and state our main results. 
The Ruelle-Perron-Fr\"obenius transfer operator $\cL_\phi$ associated to 
$f:M\to M$ and $\phi:M\to\real$ is the linear operator defined on a Banach space $X \subset C^0(M,\mathbb R)$ 
of continuous functions $\vr:M\to\real$ by
$$
\cL_\phi \vr(x) = \sum_{f(y)=x} e^{\phi(y)}\vr(y).
$$
Since $f$ is a local homeomorphism it is clear that $\cL_\phi \vr$ is continuous for every continuous $\vr$
and, furthermore, $\cL_\phi$  is indeed a bounded operator relative to the norm of uniform convergence in 
$C^0(M,\mathbb R)$ because
$
\|\cL_\phi\| \le  \deg(f) \; e^{\sup|\phi|}.
$
Analogously, $\cL_\phi$ preserves the Banach space $C^{\alpha}(M,\mathbb R)$, $0<\alpha<1$ of H\"older
 continuous observables. Moreover, it is not hard to check that $\cL_\phi$ is a bounded linear operator in the Banach 
space $C^{r}(M,\mathbb R)\subset C^0(M,\mathbb R)$ ($r\ge 1$) endowed with the norm $\|\cdot\|_r$ 
whenever $f$ is a $C^r$-local diffeomorphism and $\phi\in C^{r}(M,\mathbb R)$.
We say that the Ruelle-Perron-Frobenius operator $\cL_\phi$ acting on a Banach space $X$ has the
\emph{spectral gap property} if there exists a decomposition of its spectrum 
$\sigma(\cL_\phi)\subset \mathbb C$ as follows: $\sigma(\cL_\phi)=\{\lambda_1\}\cup \Sigma_1$ where
$\lambda_1$ is a leading eigenvalue for $\cL_\phi$ with one-dimensional associated eigenspace
and $\Sigma_1 \subsetneq \{ z\in \mathbb C : |z|<\lambda_1 \}$.

The first result is a spectral gap for the Ruelle-Perron-Frobenius 
operator in the space of H\"older continuous observables, which is enough to derive the uniqueness and
further regularity of the density of the equilibrium state with respect to the conformal measure.

\begin{maintheorem}\label{thm.spectralgap}
Let $f:M \to M$ be a 
local homeomorphism with Lipschitz continuous inverse and $\phi:M \to \R$ be a H\"older continuous potential
satisfying (H1), (H2) and (P).  Then the Ruelle-Perron-Frobenius has a spectral gap property in the space of H\"older continuous observables, there exists a unique equilibrium state $\mu$ for $f$ with respect to $\phi$ and the density $d\mu/d\nu$ is H\"older continuous. 
\end{maintheorem}

Let us mention that the previous result holds for more general compact invariant subsets $K\subset M$
(with the induced topology) also under the assumption that every point has constant number of preimages 
in $K$ and at least one preimage in the expanding region, as considered in \cite{VV10}. Since we will be interested in further extensions to differentiable dynamics as discussed below we will not prove or use this fact here.
Let us give two important consequences of the previous result.

\begin{maincorollary}\label{cor.decay}
The equilibrium state $\mu$ has exponential decay of
correlations for H\"older continuous observables:
there exists some constants $0<\tau<1$ such that
for all $\varphi\in L^1(\nu), \psi\in C^{\alpha}(M)$ there exists $K(\varphi,\psi)>0$ satisfying
\begin{equation*}
\left|\int_M (\varphi\circ f^n) \psi d\mu - \int_M \varphi d\mu\int_M \psi d\mu\right|
	\leq K(\varphi,\psi)\cdot\tau^n,
	\quad \text{for every $n\ge 1$}.
\end{equation*}
\end{maincorollary}

As a byproduct of the previous theorem we also obtain a Central Limit Theorem.

\begin{maincorollary}\label{c.CLT}
Let $\varphi$ be a H\"older
continuous function and set 
$$
\sigma_\varphi^2:=\int v^2 d\mu + 2\sum\limits_{j=1}^{\infty} v\cdot (v\circ f^j) \, d\mu,
\quad \text{ where } \quad v=\varphi-\int \varphi \, d\mu.
$$
Then $\sigma_\varphi<\infty$ and $\sigma_\varphi=0$ iff $\varphi=u\circ f - u$ for some
$u \in L^1(\mu)$. Furthermore, if $\sigma_\varphi>0$ then the following convergence on distribution
\begin{equation*}
\mu\left(x\in M: \frac{1}{\sqrt{n}}\sum\limits_{j=0}^{n-1}
\left(\varphi(f^j(x))-\int \varphi d\mu\right)\in A\right)\to \frac{1}{\sigma_\varphi\sqrt{2\pi}}
\int_A e^{-\frac{t^2}{2\sigma_\varphi^2}} dt,
\end{equation*}
holds as $n\to\infty$ for every interval $A\subset\real$.
\end{maincorollary}

The stability of the equilibrium state under deterministic perturbations is more subtle. In fact, 
the Ruelle-Perron-Frobenius operator $\cL_{f,\phi}$ acting on the space of H\"older 
continuous observables is continuous on the potential $\phi$ but in general it
\emph{may not vary continuously} with the underlying dynamics $f$, as shown in Example~\ref{ex.transfer}.
Nevertheless we could obtain further that the H\"older continuous densities of the equilibrium states with respect
to the conformal measures vary  continuously with the dynamics in the $C^0$-topology and that the topological
pressure varies continuously, which gives a nontrivial extension of the weak$^*$ stability results in \cite{VV10}. 
  
\begin{maintheorem}\label{thm.spectralgap2}
Let $\cF$ be a family of local homeomorphisms with Lipschitz inverse
and let $\mathcal W$ be some family of H\"older continuous potentials
satisfying (H1), (H2)  and (P) with uniform constants. Then the topological pressure
function $\cF \times \cW  \ni (f,\phi)  \to \Ptop(f,\phi)$ is continuous.
Moreover, the invariant density function 
\[
\begin{array}{rcl}
\cF \times \cW  &  \to & C^\alpha(M,\mathbb R) \\
   (f, \phi) & \mapsto & \frac{d\mu_{f,\phi} }{d\nu_{f,\phi}}
\end{array}
\]
is continuous whenever $C^\al(M,\mathbb R)$ is endowed with the $C^0$ topology. 
\end{maintheorem}

\subsubsection*{Stronger stability results}

Now we pay attention to the stability of the equilibrium states under both deterministic 
and a arbitrary random perturbations. To obtain stronger statistical
stability results we will admit that the dynamics is $C^r$-differentiable 
($r\ge 1$) and give a detailed study of the spectral properties for the Ruelle-Perron-Frobenius operator 
acting on the space $C^{r}(M, \mathbb R)$. Associated to $\phi\in C^r(M,\mathbb R)$ consider the condition:
\begin{itemize}
\item[(P')] $\sup\phi-\inf\phi<\vep_\phi$ \quad  and \quad $\max_{s\le r}\|D^s \phi\|_0 <\vep'_\phi$ 
\end{itemize}
for some $\vep'_\phi>0$ expressed precisely in equation~\eqref{eq.vepp} and depending on $L$, $\si$, $q$, 
$\deg(f)$, $\vep_\phi$ and $r$. This is an open condition on the set of potentials, satisfied 
by constant potentials, and a natural generalization of  condition (P) to the differentiable setting.

\begin{maintheorem}\label{Thm.Statistical}
Given an integer $r\ge 1$, let $\cF^r$ be a family of $C^r$ local diffeomorphisms and let $\mathcal W^r$ be a 
family of $C^r$-potentials satisfying (H1), (H2) and (P') with uniform constants. 
Then the topological pressure $\cF^r \times \cW^r  \ni (f,\phi)  \to \Ptop(f,\phi)$ and the invariant density 
\[
\begin{array}{rcl}
\cF^r \times \cW^r  &  \to & C^{r}(M,\mathbb R) \\
   (f, \phi) & \mapsto & \frac{d\mu_{f,\phi} }{d\nu_{f,\phi}}
\end{array}
\]
vary continuously in the $C^r$ topology. Moreover, the conformal measure function 
\[
\begin{array}{rcl}
\cF^r \times \cW^r  &  \to & \cM(M) \\
   (f, \phi) & \mapsto & \nu_{f,\phi}
\end{array}
\]
is continuous in the weak$^*$ topology. In consequence, the equilibrium measure
$\mu_{f,\phi}$ varies continuously in the weak$^*$ topology.
\end{maintheorem}

Finally we will describe our results on the stability of the spectra of the
Ruelle-Perron-Frobenius operator under random perturbations.
Given $r\in\mathbb N$, and families $\cF^r$ of local diffeomorphisms 
and $\cW^r$ of $C^r$-observables satisfying (H1), (H2) and (P') with uniform constants, 
a \emph{random perturbation} of $f\in\cF$ is a family $\theta_\vep$, $0< \vep \le 1$ of probability
measures in $\cF^r\times \cW^r$ such that there exists a family $V_\vep(f,\phi)$, $0<
\vep \le 1$ of neighborhoods of $(f,\phi)$, depending monotonically on
$\vep$ and satisfying
$$
\supp\theta_\vep \subset V_\vep(f,\phi) \quad\text{and}\quad
\bigcap_{0<\vep\le 1} V_\vep(f,\phi) =\{(f,\phi)\}.
$$
This dynamics can be codified by considering the skew product map
\[
\begin{array}{rcl}
F: \cF^\N \times M  &  \to & \cF^\N \times M \\
   (\un f, x) & \mapsto & (\si(\un f),f_1(x))
\end{array}
\]
where $\un f=(f_1, f_2, \ldots)$ and $\si: \cF^\N \to \cF^\N$ is the
shift to the left. 
Associated to this random dynamical system consider the \emph{integrated
Ruelle-Perron-Frobenius operator} $\cL_\vep$ given by
\begin{equation}\label{integratedRPF}
\cL_\vep \vr(x) =\int (\cL_{f,\phi} \vr)(x) \; d\theta_\vep(f).
\end{equation}
We say that $(f,\phi)$ has \emph{$C^r$-spectral stability} under the random perturbation if the operator 
$\cL_\vep$ in the Banach space $C^{r}(M,\mathbb R)$ has the spectral gap property and the leading eigenvalue
$\lambda_\vep$ and associated eigenfunction $h_\vep$ vary continuously with $\vep$ 
and accumulate, as $\vep\to 0$,  respectively on the leading eigenvalue and eigenfunction of the 
unperturbed operator.  We prove the following spectral stability under random perturbations.

\begin{maintheorem}\label{thm.spectral.stability}
Let $(\theta_\vep)_\vep$ be any random perturbation of $(f,\phi)  \in \cF^r\times \cW^r$.  
Then $(f,\phi)$ has $C^r$-spectral stability under the random perturbation
$(\theta_\vep)_{\vep}$. 
\end{maintheorem}

Some comments are in order. Weaker stochastic stability results were previously obtained in \cite{VV10}
under a non-degeneracy assumption. Namely, assuming that all $f\in \cF$ are non-singular with respect to
a fixed conformal measure it follows that there are stationary measures $\mu^\vep$ absolutely continuous with 
respect to the conformal measure $\nu$ and that converge to the equilibrium state $\mu$
in the weak$^*$ topology as the noise level $\vep$ tends to zero. Here we obtain spectral stability under 
arbitrary random perturbations.

\section{Preliminaries}\label{s.preliminaries}

In this section we provide some preparatory results needed for the proof of the
main results. Namelly, we study the combinatorics of the orbits, hyperbolic times and
some pressure estimates.

\subsection{Combinatorial estimates for orbits}

Here we give a description of the orbits of points according to the visit to the possibly not expanding
region $\cA$ using an auxiliary partition $\cP$ built using (H2).

\begin{lemma}\label{l.partitionP}
There exists a partition $\cP$ of $M$ of domains of injectivity for $f$ with cardinality at most $\sharp \cU$ 
and such that  $\cup\{U\in \cU: U \cap \cA\neq \emptyset\}=\cup\{P\in\cP: P \cap \cA\neq \emptyset\}$. In particular
there are at most $q<\deg(f)$ elements of $\cP$ that cover $\cA$.
\end{lemma}

\begin{proof}
Pick an enumeration $\{U_i\}$ of the open covering $\cU$ given by (H2) in such a way that 
the region $\cA$ is covered by the first $q$ elements of $\cU$. Consider the partition $\cP$
given by $P_1=U_1$ and, recursively, $P_{i+1}=U_{j+1}\setminus (\cup_{j=1}^i P_j) $ for $i=1\dots \#\cU-1$.
It is clear that $\sharp \cP \leq \sharp \cU$.  Moreover, $f\mid_{P_i}$ is injective for every nonempty $P_i$ since 
by construction $P_i\subset U_i$ and
$$
\bigcup\{U\in \cU: U \cap \cA\neq \emptyset\}
	= \bigcup_{j=1}^q U_j 
	= \bigcup_{j=1}^q P_j 
	=\bigcup\{P\in\cP: P \cap \cA\neq \emptyset\}.
$$
Since the last statement in the lemma is immediate from the construction this finishes the proof of the lemma.
\end{proof}

Since the region $\cA$ is contained in $q$ elements of the partition
$\cP$ we can assume without any loss of generality that $\cA$ is
contained in the first $q$ elements of $\cP$. For all $x$ we can associate 
an \emph{itinerary} $\underline i(x) \in (i_0,\dots,i_{n-1}) \in \{1, \dots, \#P\}^n$ 
by $i_j=\ell$ if and only if $f^j(x) \in P_{\ell}$.
Given $\ga\in(0,1)$ and $n\ge 1$, let us consider also the set $I(\ga,n)$ of all itineraries
$(i_0,\dots,i_{n-1})$ so that $\# \{0\le j
\le n-1: i_j\le q\} > \gamma n$.

\begin{lemma}\label{l.combinatorio}
Given $\vep>0$ there exists $\ga_0 \in (0,1)$ such that
\begin{equation*}
c_\ga := \limsup_{n\to\infty} \frac{1}{n} \log \#{I(\ga,n)}
	< \log q +\vep
\end{equation*}
for every $\ga \in (\ga_0,1)$.
\end{lemma}

\begin{proof}
See \cite[Lemma~3.1]{VV10}.
\end{proof}

We are in a position to state our precise condition on the constant
$L$ in assumption (H1) and the constant $c$ in the definition of
hyperbolic times. First note that if $\sup\phi-\inf \phi<\log\deg(f)-\log q$ as in Theorem~\ref{thm.VV}
then it follows from Lemma~\ref{l.combinatorio} that one may find $\gamma<1$ such that
$c_\gamma<\log\deg(f)-\sup\phi+\inf \phi$.
We assume that $L$ is close enough to $1$, and $c>0$ and $0<\vep_\phi< \log\deg(f) - \log q$ 
are so that
\begin{equation}\label{eq. relation expansion}
\sigma^{-(1-\gamma)} L^\gamma<e^{-2c} < 1
\end{equation}
and
\begin{equation}\label{eq. relation potential}
e^{\vep_\phi}\cdot\left(\frac{(\deg(f)-q) \sigma^{-\alpha} + q L^\alpha [1+(L-1)^\alpha] }{\deg(f)} 
\right)<1
\end{equation}
The first condition is to guarantee the existence of infinitely many hyperbolic times with respect to the reference
measure in the proposition below. The second technical condition roughly means that $f$ has some average backward contraction and will be used to obtain the invariance of a cone of functions under the 
Ruelle-Perron-Frobenius operator in Proposition~\ref{t.cone.invariance}.

\subsection{Ruelle-Perron-Frobenius operators and conformal measures}\label{Conformal Measure}

Recall that the Ruelle-Perron-Frobenius transfer operator $\cL_\phi:
C^0(M,\mathbb R) \to C^0(M,\mathbb R)$ associated to $f:M\to M$ and $\phi:M\to\real$ is the
linear operator defined on the space $C^0(M,\mathbb R)$ of continuous functions
$\vr:M\to\real$ by
$$
\cL_\phi \vr(x) = \sum_{f(y)=x} e^{\phi(y)}\vr(y).
$$
In fact, $\cL_\phi \vr$ is continuous since $f$ is a local homeomorphism and  $\vr$ is continuous.
Moreover, it is not hard to check that $\cL_\phi$ is a bounded operator, relative to the norm of uniform
convergence in $C^0(M,\mathbb R)$ and $\|\cL_\phi\| \le  \deg(f) \; e^{\sup|\phi|}$.
Consider also the dual operator $\cL^*_\phi:\cM(M)\to\cM(M)$ acting on
the space $\cM(M)$ of Borel measures in $M$ by
$
\int \vr \, d(\cL_\phi^*\eta) = \int (\cL_\phi \vr) \, d\eta
$
for every $\vr \in C^0(M,\mathbb R)$. Let $r(\mathcal L_\phi)$ be the
spectral radius of $\cL_\phi$.  In our context  conformal measures associated to
the spectral radius always exist as stated in the next proposition, whose proof 
can be found in the proofs of Theorem~B and Theorem~4.1 in \cite{VV10}.

\begin{proposition}\label{p.conformal.measure}
If $f$ is topologically exact and satisfies (H1), (H2) and $\phi$ satisfies $\sup\phi-\inf\phi<\log\deg(f)-\log q$ 
then there exists an expanding 
conformal  measure such that $\cL_\phi^* \nu=\la \nu$ and $\supp(\nu)= \overline H$,
where $\la=r(\cL_\phi)\ge \deg(f) e^{\inf \phi}$. Moreover, $\nu$ is a non-lacunary Gibbs measure
and has a Jacobian with respect to $f$ given by $J_{\nu} f =\la e^{-\phi}$.
\end{proposition}

Just for completeness let us mention that one key ingredient is that our assumptions guarantee
 we obtain volume expansion with respect to the
conformal measure, that is, $J_\nu f (x) \geq \deg(f) e^{\inf \phi-\sup\phi} > e^{c_\ga}$ for all $x\in M$.
This is enough to guarantee that $\nu$-almost every point
spends at most a fraction $\ga$ of time inside the domain $\cA$
where $f$ may fail to be expanding. Notice also that $\lambda=\int\cL_\phi 1 d\nu$.

Finally, we collect the main estimates concerning the pressure of the invariant sets $H$ and $H^c$, 
which play  a key role in the construction of equilibrium states. 

\begin{proposition}\label{p.pressure}
$P_{\text{top}}(f,\phi)
	= P_{H}(f,\phi)
	= \log \la
	> P_{H^c}(f,\phi)$,
where $\lambda$ denotes the spectral radius of the Ruelle-Perron-Frobenius $\cL_\phi$
acting on the space of continuous observables. In consequently, any equilibrium state is an 
expanding measure. 
\end{proposition}

\begin{proof}
See Proposition~6.1, Lemma~6.4 and Lemma~6.5 in \cite{VV10}.
\end{proof}

\subsection{Regularity of the observables}\label{s.regularity}

Here we study a relation between H\"older and locally H\"older continuous functions.  
We say that $\varphi: M\to \mathbb R$ is \emph{$(C,\alpha)$-H\"older continuous in balls of radius $\delta$} if
$$
|\varphi(x)-\varphi(y)|\leq C d(x,y)^\alpha
$$
for every $y\in B(x,\delta)$ and $x\in M$. Our first auxiliary lemma for
the regularity of observables is as follows.

\begin{lemma}\label{leholoc}
Given $1\le \zeta \le 2$ and $\delta> 0$, if $\varphi: M\to \mathbb R$ is
$(C,\alpha)$-H\"older continuous in balls of radius $\delta$ then it is $(C (1+ r^{\alpha}), \alpha)-$H\"older
continuous in balls of radius $(1+ r) \delta \leq \zeta \delta$, with $0<r\le 1$.
\end{lemma}

\begin{proof}
Since $M$ is connected then given $y, z\in M$ so that $d(y,z)<(1+r)\delta$ by considering
a geodesic arc connecting $y$ and $z$ in $M$ there exists $w$ so that $d(z, w)= \delta$ and 
$d(w, y)< r d(z, w)< \delta$. Therefore
\begin{align*}
|\vr(z)- \vr(y)|
	 & \leq |\vr(z) - \vr(w)|+ |\vr(w) - \vr(y)| \leq C d(z, w)^\alpha+ C d(w, y)^\alpha \\
	 & \leq C (1+ r^\alpha) d(z, w)^\alpha \leq  C (1+ r^\alpha) d(z, y)^\alpha,
\end{align*}
which proves the lemma.
\end{proof}

The next lemma asserts that every locally H\"older continuous observable is indeed H\"older continuous.
Moreover, we give an estimate for the H\"older constant.

\begin{lemma} \label{lehoglob}
Let $N$ be a compact and connected metric space.
Given $\delta> 0$ there exists $m \geq 1$ (depending only on $\delta$) such that the following holds:
if $\varphi:N \to \mathbb R$ is $(C, \alpha)$-H\"older continuous in balls of radius $\delta$ then
it is $(C m, \alpha)$-H\"older continuous.
\end{lemma}

\begin{proof}
Fix $\delta>0$ and let ${\mathcal B}= \{B(x_i, \delta/3)\}_{i=1\dots s}$ be a finite covering of $N$.
We can assume, without loss of generality, that $x_j \in B(x_{j+ 1}, \delta)$ for every  $j= 1, \dots, s- 1$.
Our hypothesis guarantee that if $x, w \in N$ with $d(x, w)< \delta$ we have $|\varphi(x)-\varphi(w)| \leq C d(x, w)^\alpha$. Hence, if $d(x, w) \geq \delta$ then it is not hard to use the triangular inequality to get
$$
|\varphi(x)-\varphi(w)|
	\leq (s+ 2)  C \delta^\alpha
	\leq C (s+ 2) \, d(x, w)^\alpha.
$$
Thus it is enough to take $m \geq s+ 2$ in the lemma.
\end{proof}

\subsection{Positive operators and cones}\label{s.cones}

In this subsection we shall recall some results concerning the theory of projective metrics on cones
and positive operators due to G. Birkhoff. Despite the great generality of this theory we shall concentrate
on cones and positive operators on Banach spaces. We refer the reader to~\cite{Li95, Ba00} for
detailed presentations.

Let $\cB$ be a Banach space. A subset
$\Lambda\subset \cB-\{0\}$ is a \emph{cone} if $r\cdot v\in\Lambda$ for all
$v\in\Lambda$ and $r\in\real^+$. The cone $\Lambda$ is \emph{closed} if
$\overline\Lambda=\Lambda\cup\{0\}$, and $\Lambda$ is \emph{convex} if $v+w\in\Lambda$ for all
$v,w\in\Lambda$.
Notice that a convex cone $\Lambda$ with $\Lambda\cap (-\Lambda) = \emptyset$ determines a
\emph{partial ordering} $\preceq$ on $\cB$ given by:
\begin{equation*}
w\preceq v \; \text{ iff } \; v-w\in\Lambda\cup\{0\}.
\end{equation*}
In the sequel, our cones $\Lambda$ are assumed to be closed, convex and
$\Lambda\cap (-\Lambda)=\emptyset$.
Given a cone $\Lambda$ and two vectors $v,w\in\Lambda$, we define
$\Theta (v,w)=\Theta_{\Lambda}(v,w)$ by
\begin{equation*}
\Theta (v,w)=\log\frac{B_{\Lambda}(v,w)}{A_{\Lambda}(v,w)},
\end{equation*}
where $A_{\Lambda}(v,w)=\sup\{r\in\real^+: \; r\cdot v\preceq w\}$ and
$B_{\Lambda}(v,w)=\inf\{r\in\real^+: \; w\preceq r\cdot v\}$. The
(pseudo-)metric $\Theta$ is called the \emph{projective metric} of $\Lambda$ (or
\emph{$\Lambda$-metric} for brevity). Defining the equivalence
relation $v\sim w$ iff $w=r\cdot v$ for some $r\in\real^+$, then $\Theta$
induces a metric on the quotient $\Lambda / \sim$.
The following key result  is due to Birkhoff, which can be found e.g.
in \cite[Proposition~2.3]{Vi97}.

\begin{theorem}\label{t.Birkhoff}
Let $\Lambda_i$ be a closed convex cone (with $\Lambda_i\cap (-\Lambda_i) =
\emptyset$) in a Banach space $\cB_i$, for $i=1,2$. If $\cL:\cB_1\to\cB_2$ is a
linear operator such that $\cL(\Lambda_1)\subset\Lambda_2$ and 
$\De=\text{diam}_{\Theta_{\Lambda_2}}(\cL\Lambda_1)<\infty$ then
\begin{equation*}
\Theta_{\Lambda_2} (\cL v,\cL w)
	\leq \left( 1 - e^{-\De} \right) \cdot \Theta_{\Lambda_1} (v,w),
\end{equation*}
for any $v,w\in\Lambda_1$.
\end{theorem}

In consequence of the previous theorem, if the diameter of the cone $\cL(\Lambda_1)$ is
finite in $\Lambda_2$ then $\cL$ is a contraction in the projective metric which enables us to prove
that it admits a unique fixed point.

\subsection{Combinatorial lemma on preimages matching}

Here we establish an auxiliary lemma to bound for the distance of preimages associated to 
different functions in $\cF$ which will play a key role in the proof of the stability results. Let 
$V_\vep(f)\subset \cF$ be an open neighborhood of $f\in\cF$.

\begin{lemma}\label{l.preimagesmatching}
Given $n\ge 1$, $\un f, \un g  \in \cF^{\mathbb N}$ and $x,y\in M$ there exists bijection between 
the sets of preimages $\{z\in M : \un f^n(z)=x\}$ and $\{z\in M : \un g^n(z)=y\}$. Moreover, for every $n\in\mathbb N$
there exists $\vep(n)>0$ such that for every $0<\vep\le \vep(n)$ the distance between paired $n$-preimages 
is such that if $d(x,y)<\vep$ and $\un g \in V_\vep(\un f)$ then
$$
d (  x^{(n)}_i,  y^{(n)}_i  ) 
	\leq L^n \, d(x,y) + \sum_{j=1}^{n} L^{n-j+1} \|f_j-g_j\|_\al.
$$
for every $i=1\dots \deg(f)^n$.
\end{lemma}

\begin{proof}
Let $\hat\cU$ be a finite open cover by balls obtained using domains of invertibility for $f$ and let $2\hat\de$
be the Lebesgue number of the covering $\hat\cU$. If $\vep>0$ is small enough the constant $2\hat\de$ can 
be taken uniform for every $\tilde f\in V_\vep(f)$. 
Let $x,y\in M$ satisfy $d(x,y)<\vep$ and take $\un f =(f_i)_{i\in \mathbb N}$ and $\un g =(g_i)_{i\in \mathbb N}$
with $\un g \in V_\vep(\un f)$.  We will prove the result recursively.

First notice that the sets $\{z\in M : f_1(z)=x\}$ and $\{z\in M : g_1(z)=y\}$ have the same 
cardinality $\deg(f)$ and thus there exists a one-to-one correspondance. Moreover, reducing $\vep>0$ if
necessary, we obtain that the paired enumerations  $\{x_i\}$ and $\{y_i\}$ of such elements verify
\begin{align*}
d (g_1(y_i) , g_1(x_i) ) 
	& = d( y, g_1(x_i) )
	\leq d( y, x ) + d(x, g_1(x_i)) \\
	& = d( y, x ) + d(f_1(x_i), g_1(x_i)) 
	 \leq d( y, x ) + \|f_1-g_1\|_\al 
	<\hat\de
\end{align*}
Since $g_1 \in \cF$ then it satisfies (H1), (H2) and so $d (x_i,y_i ) \leq L \, [ d( y, x ) + \|f_1-g_1\|_\al] $ for
every $i$.
The same argument as above applied to the pairs $x_i=f_2(x_j^{(2)})$ and $y_i=g_2(y_j^{(2)})$
proves that $d (g_2(y^{(2)}_j) , g_2(x^{(2)}_j) ) \leq  d( x_i,y_i ) + \|f_2-g_2\|_\al$
and, consequently,
\begin{align*}
d ( y^{(2)}_j  ,  x^{(2)}_j  ) 
	& \leq L  [ d( x_i,y_i ) + \|f_2-g_2\|_\al] \\
	& \leq L^2 [ d(x,y) + 	\|f_1-g_1\|_\al] + L \|f_2-g_2\|_\al \\
	& = L^2  d(x,y) + L^2 \|f_1-g_1\|_\al + L \|f_2-g_2\|_\al,
\end{align*}
which can be taken also smaller than $\hat \de$ provided that we reduce $\vep$.
Using the same reasoning recursively, if $d(x,y)<\vep(n)$ small so that the corresponding paired enumerations
of preimages $(x^{(k)}_i)$ and $(y^{(k)}_i)$, $i=1\dots \deg(f)^k$ in the sets $\{z\in M : \un f^k(z)=x\}$ and 
$\{z\in M : \un g^k(z)=y\}$ are $\hat \delta$-close for every $1\le k \le n-1$. Moreover, applying the previous reasoning
it follows that
\begin{equation*}
d (  x^{(n)}_i,  y^{(n)}_i  ) 
	\leq L^n \, d(x,y) + \sum_{j=1}^{n} L^{n-j+1} \|f_j-g_j\|_\al
\end{equation*}
as claimed. This finishes the proof of the lemma.
\end{proof}

We also get a simple expression for the distance of $n$-preimages associated to the same
close functions in $\cF$ and the same base point in $M$. 

\begin{corollary}\label{cor.preimages}
Given $n\in \mathbb N$ there exists $\vep(n)>0$ such that for any $f_1,f_2\in \cF$ with $\|f_1-f_2\|_\al<\vep(n)$ 
the following property holds: given $x\in M$ and paired preimages $(x^{(n)}_{1i})$ and $(x^{(n)}_{2i})$ by $f_1^n$ and $f_2^n$, respectively, then 
$$
d ( x^{(n)}_{1i} , x^{(n)}_{2i}  ) 
	\leq n L^n  \|f_1-f_2\|_\al
	\quad \text{for all $i$.}
$$
\end{corollary}

\begin{proof}
This is a direct consequence of the previous lemma, by considering $x=y$ and the sequences of functions $\un f=(f_1,f_1,f_1,\dots)$ and $\un g=(f_2,f_2,f_2,\dots)$ in $V_\vep(f)^{\mathbb N}$ with $\vep$ small.
\end{proof}

We finish this section by proving that paired preimages associated to any close points have similar
behaviour with respect to the region $\cA$. More precisely, 

\begin{lemma}\label{l.aux}
Let $f$ satisfy assumptions (H1) and (H2). Then there exists $\de>0$ so that for every ball $B$ of radius $\de$ 
has at most $q<\deg(f)$ connected components  in $f^{-1}(B)$ that intersect $\cA$. In particular, if $d(x,y)<\de$ then there are at most $q$ pairs of paired preimages by $f$ associated to $x$ and $y$ that belong to $\cA$.
\end{lemma}

\begin{proof}
Assume that $\de_0>0$ is small so that every inverse branch is well defined in a ball of radius $\de_0$.
Since $\sharp (\{f^{-1}(x)\}\cap \cA)\leq q$ then for every $x\in M$ there exists $0<\de_x<\de_0$ so that 
$f^{-1}(B(x,\de_x))$ has at most $q$ connected components that intersect $\cA$. By compactness of $M$ 
pick a finite subcover $\cB=\{B(x_i,\de_i)\}_{i\in I}$, set $2\de$ to be Lebesgue number of $\cB$ and 
assume, without loss of generality, that $\de<\de_0$.
Therefore, by construction, given any ball $B$ of radius $\de$ it follows that $B \subset B(x_i,\de_i)$ for 
some ${i\in I}$. In consequence, the number of connected components satisfy
$$
\sharp c.c. (f^{-1}(B)\cap \cA)
	\leq \sharp c.c. (f^{-1}(B(x_i,\de_i))\cap \cA)
	\leq q
$$ 
This finishes the proof of the lemma.
\end{proof}

\section{Ruelle-Perron-Frobenius operator in $C^\alpha(M,\mathbb R)$:
		 \\Spectral gap and statistical consequences}\label{s.gapHolder}

In this section we prove that the action of the transfer operator in the space of H\"older continuous
observables has the spectral gap property. In consequence, we provide an alternative 
proof for the existence and uniqueness of equilibrium states as well as 
further statistical properties: exponential decay of correlations and central limit theorem.  
We also get that the densities of the unique equilibrium state with respect to the conformal measures
are H\"older and vary continuously in a uniform way with the dynamical system.  Finallly, the topological
pressure also varies continuously in this non-uniformly expanding setting.

\subsection{Invariant cones for the transfer operator in $C^\alpha(M,\mathbb R)$}\label{s.s.cones}

To prove that the Ruelle-Perron-Frobenius operator has a spectral gap in the space of H\"older
continuous observables one first introduce some notations. Recall that the H\"older constant of 
$\vr\in C^\alpha(M,\mathbb R)$ is
$$
|\vr|_{\alpha}=\sup\limits_{x\neq y}\frac{|\vr(x)-\vr(y)|}{d(x,y)^{\alpha}}
$$
and set $|\vr|_{\alpha,\delta}$ as the least constant $C>0$ such that $|\vr(x)-\vr(y)|\le C d(x,y)^\alpha$ for  all
points $x,y$ such that $d(x,y)<\delta$. Now, consider the cone of locally H\"older continuous observables
$$
\Lambda_{\kappa,\delta}
	=\Big\{\vr\in C^0(M,\mathbb R) : \vr>0 \text{ and } \frac{|\vr|_{\alpha,\delta}}{\inf \vr} \leq \kappa \Big\}.
$$

Throughout, let $\de>0$ be fixed and given by Lemma~\ref{l.aux}. Fix also $m$ given by 
Lemma~\ref{leholoc} associated to balls of radius $\de$.
We are now in a position to state the precise condition on the constants $\vep_\phi$ and
$\vep'_\phi$ on (P) and (P') respectively. Then taking into account \eqref{eq. relation potential} we assume:
\begin{equation}\label{eq.vep}
e^{\vep_\phi}\cdot\left(\frac{(\deg(f)-q) \sigma^{-\alpha} + q L^\alpha [1+(L-1)^\alpha] }{\deg(f)} \right)
	+ \vep_\phi 2m L^\al \diam(M)^\al
	<1
\end{equation}
and
\begin{equation}\label{eq.vepp}
[1+\vep'_\phi ] \cdot e^{\vep_\phi} \cdot \left(\frac{(\deg(f)-q) \sigma^{-\alpha} + q L^\alpha [1+(L-1)^\alpha] }{\deg(f)} \right)
	<1
\end{equation}
Notice that having \eqref{eq. relation potential} it is possible to consider $\vep'_\phi$ satisfying the later 
condition. Our main result in this section is as follows.

\begin{theorem}\label{t.cone.invariance}
Assume that $f$ satisfies (H1), (H2) and that $\phi$ satisfies (P). Then
there exists $\delta>0$ and $0< \hat \la< 1$ such that $\cL_\phi(\Lambda_{\kappa,\delta}) \subset \Lambda_{\hat\la \kappa,\delta} $ for every large positive constant $\kappa$.
\end{theorem}

\begin{proof}
Take $\kappa>0$ and let  $\vr\in \Lambda_{ \kappa,\delta}$ be given. Moreover, given $x\in M$ we consider the
set $(x_j)_{j=1\dots\deg(f)}$ of the preimages by $f$ of the point $x$, that is $f(x_j)=x$, and let $K=|\varphi|_{\al,\de}$ be the $\alpha$-H\"older
constant of $\vr$ on balls of radius $\delta$. We will prove that there exists $0<\hat\lambda<1$ such that 
$\cL_\phi \vr \in \Lambda_{\hat\la \kappa,\delta}$ provided that $\kappa$ is large enough. Indeed, if $d(x,w)<\delta$ then
\begin{align}
\frac{| \cL_\phi \vr(x)- \cL_\phi \vr(w)|}{\inf_{z \in M}\{\cL_\phi \vr(z)\}d(x, w)^\alpha }
	& \leq \frac{\sum_{j= 1}^{\deg(f)} |\vr(x_j) e^{\phi(x_j)}- \vr(w_j) e^{\phi(w_j)}|}{{\deg(f)} \cdot e^{\inf\phi} \inf \vr \cdot d(x, 	w)^\alpha} \nonumber \\
	& \leq \frac{\sum_{j= 1}^{\deg(f)} |e^{\phi(x_j)}(\vr(x_j) - \vr(w_j))|}{{\deg(f)} \cdot e^{\inf\phi} \inf \vr \cdot d(x, w)^\alpha}
	\label{eq.spec1}\\
	& +  \frac{\sum_{j= 1}^{\deg(f)} | \vr(w_j) (e^{\phi(x_j)}-  e^{\phi(w_j)})|}{{\deg(f)} \cdot e^{\inf\phi} \inf \vr \cdot
	d(x, w)^\alpha} \label{eq.spec2}
\end{align}
We subdivide the sum in \eqref{eq.spec1} according to the possible backward contraction of the inverse branches of
$f$. Using Lemma~\ref{leholoc} we obtain that $\varphi$ is $\tilde K$-H\"older on balls of radius $L\delta$ with
$\tilde K=K (1+(L-1)^\alpha)$. Since $K\le \kappa \inf g$ and $\sup\phi -\inf\phi<\vep_\phi$ we get that 
\eqref{eq.spec1} is bounded from above by
$$
e^{\vep_\phi} \;  \frac{(\deg(f)-q) \sigma^{-\alpha}+ q L^\alpha [1+(L-1)^\alpha]}{{\deg(f)}} \kappa.
$$
For estimating \eqref{eq.spec2} we first note that since $\vr$ is H\"older continuous then $\sup \vr \le \inf \vr + m |\vr|_{\al,\de}\diam(M)^\alpha$. Therefore, using $\vr\in\Lambda_{\kappa,\delta}$ we get 
\begin{align*}
\eqref{eq.spec2}
	& \leq \frac{\sup \vr \cdot |e^\phi|_\alpha \cdot L^\al}{\inf \vr\; \cdot e^{\inf\phi}}
	 \leq \frac{\inf \vr + m \|\vr\|_{\al,\de}\diam(M)^\alpha}{\inf \vr}  \frac{|e^\phi|_\alpha}{e^{\inf\phi}} L^\al \\
	 &  \leq \frac{|e^\phi|_\alpha}{e^{\inf\phi}} L^\al \left[1 + m \kappa \diam(M)^\alpha \right]
	 \leq \frac{|e^\phi|_\alpha}{e^{\inf\phi}} 2 m L^\al \kappa \diam(M)^\alpha
\end{align*}
Using (P) we have $|e^\phi|_\alpha<\vep_\phi e^{\inf\phi}$ and our previous choice of $\vep_\phi$ yields that
 $\|\cL_\phi \vr\|_{\al,\de} \leq \hat \la \kappa \;\inf (\cL_\phi \vr)$. This completes the proof of the theorem.
\end{proof}

Observe that assumption (P) can be rewritten as $\sup(\phi)-\inf(\phi)<\vep_\phi$ 
and $e^\phi \in \Lambda_{\vep_\phi}$. To prove that the cone has finite diameter in  $\Lambda_{\kappa,\delta}$ 
we compute an explicit expression for the projective metric. 

\begin{lemma}\label{l.cone-metric}
The $\Lambda_{\kappa,\de}$-cone metric $\Theta_\kappa$ is given by
$\Theta_\kappa(\vr, \psi)=\log\frac{B_\kappa(\vr, \psi)}{A_\kappa(\vr,\psi)}$, where
$$
A_\kappa(\vr, \psi)=\inf\limits_{0<d(x,y)<\de, \; z\in M }
\frac{\kappa|x-y|^{\al}\psi(z)-\left(\psi(x)-\psi(y)\right)}
{\kappa|x-y|^{\al}\vr(z)-\left(\vr(x)-\vr(y)\right)} ,
$$
and
$$
B_\kappa(\vr,\psi)=\sup\limits_{0<d(x,y)<\de, \; z\in M }
\frac{\kappa|x-y|^{\al}\psi(z)-\left(\psi(x)-\psi(y)\right)}
{\kappa|x-y|^{\al}\vr(z)-\left(\vr(x)-\vr(y)\right)} .
$$
\end{lemma}

\begin{proof}
By definition, $A \vr\preceq \psi$ if and only if $\psi(x)-A \vr(x)\geq 0$ for every $x\in M$ and 
$\| \psi-A \vr \|_{\al,\de}\leq \kappa\inf (\psi-A \vr)$. In particular one gets
\begin{equation*}\label{eq.metric1}
A\leq \min\Big\{ 
	\inf\limits_{x} \frac{\psi(x)}{\vr(x)}, \ \
	\inf\limits_{0<d(x,y)<\de, \; z\in M } \frac{\kappa |x-y|^{\al}\psi(z)-\left(\psi(x)-\psi(y)\right)}{\kappa |x-y|^{\al}\vr(z)-\left(\vr(x)-\vr(y)\right)}
	\Big\}.
\end{equation*}
We will prove that minimum in the right hand side is always attained by the second term. 
Pick $x_0\in M$ such that $\inf\limits_x \frac{\psi(x)}{\vr(x)} =\frac{\psi(x_0)}{\vr(x_0)}$. 
Then it is immediate that
$$
\lim_{x \to x_0} \frac{\kappa |x- x_0|^{\al}\psi(x_0)-\left(\psi(x)-\psi(x_0)\right)} {\kappa |x- x_0|^{\al}\vr(x_0)-\left(\vr(x)-\vr(x_0)\right)}
	\leq \frac{\psi(x_0)}{\varphi(x_0)},
$$
which guarantees that
$$
A_\kappa(\vr, \psi)
	= \inf\limits_{0<d(x,y)<\de, \; z\in M}
	\frac{\kappa |x-y|^{\al}\psi(z)-\left(\psi(x)-\psi(y)\right)}{\kappa |x-y|^{\al}\vr(z)-\left(\vr(x)-\vr(y)\right)}.$$
Similar computations lead to the expression for $B_\kappa(\vr,\psi)$.

\end{proof}

\begin{proposition} \label{propdiam}
Given $0< \hat\la< 1$,  the cone ${\Lambda}_{\hat\la \kappa,\delta}$ has finite ${\Lambda}_{\kappa,\delta}$-diameter.
\end{proposition}

\begin{proof}
For all $\vr \in {\Lambda}_{\hat \lambda\kappa,\de}$ by definition we have
$|\vr|_{\alpha,\de} \leq \hat\la \kappa \inf \varphi$ and, consequently,
$\sup \varphi \leq [1+ m \hat\lambda \kappa (\diam M)^\alpha] \inf \varphi$.
So, using the previous expression for the projective metric, given $\vr, \psi \in {\Lambda}_{\hat\la \kappa,\delta}$ one 
can easily check that
\begin{align*}
\Theta_\kappa( \vr, \psi)
	& \leq \log \Big(\frac{\kappa \cdot \sup \vr + \hat\la \kappa \inf \vr}{\kappa \cdot \inf \vr - \hat\la \kappa \inf \vr } \cdot \frac{\kappa \cdot \sup \psi +\hat\la \kappa \inf \psi}{\kappa \cdot \inf \psi - \hat\la \kappa \inf \psi }\Big) \\
	& \leq \log \Big(\frac{(\kappa (1+ m \hat\lambda \kappa (\diam M)^\alpha)(1 + \hat\la ) \inf \vr}{\kappa \cdot (1- \hat\la)  \cdot \inf \vr } \Big) \\
	& +  \log \Big(\frac{(\kappa (1+ m \hat\lambda \kappa (\diam M)^\alpha)(1 + \hat\la ) \inf \psi}{\kappa \cdot (1- \hat\la)  \cdot \inf \psi} \Big) \\
	& \leq 2\log\left(\frac{1+ \hat\la}{1- \hat\la}\right)+ 2\log\big(1+ m c \diam(M)^\alpha \big)
\end{align*}
for some positive constant $c$. This implies the finite $\Theta_\kappa$-diameter of ${\Lambda}_{\hat\la\kappa,\de}$ and finishes the proof of the lemma.
\end{proof}

\subsection{Consequences of the spectral gap in $C^\al(M,\mathbb R)$}\label{s.consequences}

Now we shall deduce the existence of equilibrium states and some of their ergodic properties.

\subsubsection{Existence of equilibrium states}\label{s.exist}

Using the spectral gap property in the space of H\"older continuous
observables we get the existence of a unique H\"older continuous invariant density $h$.

\begin{proposition}\label{p.densidade}
There exists a unique density $h\in C^\alpha(M,\mathbb R)$ such that $\cL_\phi h=\lambda h$.
In particular, $\mu=h \nu$ an equilibrium state for $f$ with respect to $\phi$. 
Finally, the density $d\mu/d\nu$ is bounded away from zero and infinity and H\"older continuous.
\end{proposition}

\begin{proof}
Consider the normalized operator $\tilde \cL_\phi= \la^{-1}\cL_\phi$, where $\la$ is the spectral radius of $\cL$ and
write ${\Lambda}^+$ for the cone of strictly positive continuous functions on $M$.
Since $\La_{\kappa,\de}\subset \La^+$ then the projective metrics satisfy $\Theta^+(\vr, \psi)\leq \Theta_\kappa(\vr, \psi)$
for any $\vr, \psi \in \La_{\kappa,\de}$, where
$$
\Theta^+(\vr, \psi)= \log \left(\frac{\sup_{x \in M}\{\vr(x)/ \psi(x)\}}{\inf_{y \in M}\{\vr(y)/ \psi(y)\}}\right).
$$
By the previous proposition, $\tilde \cL_\phi(\La_{\kappa,\de})$ has finite diameter in $\La_{\kappa,\de}$ for any sufficiently large $\kappa$. Therefore, as discussed at the end of Subsection~\ref{s.cones}, $\tilde \cL_\phi$ is a contraction in the $\Theta_\kappa$-metric and there exists $0< \tau< 1$ such that for any 
$\vr, \psi \in \La_{\kappa,\de}$ and $n, k\ge 1$
\begin{equation}\label{eq.cont}
\Theta^+(\tilde \cL_\phi^{n+k}(\vr), \tilde \cL_\phi^n(\psi)) 
	\leq \Theta_\kappa(\tilde \cL_\phi^{n+k}(\vr), \tilde \cL_\phi^n(\psi))
	\leq \Delta \tau^n,
\end{equation}
where $\Delta$ is the $\Theta_\kappa$-diameter of the cone $\La_{\hat\la \kappa,\delta}$.
This proves that $(\tilde \cL_\phi^{n} \varphi)_n$ is a Cauchy sequence in the projective metric.
For the reference measure $\nu$ we have that
$
\int \tilde \cL_\phi \vr \, d\nu= \int \vr \,d\nu,
	\; \forall \vr \in C^0(M, \re).
$
Given $\vr \in \La_{\kappa,\de}$ with $\int \vr \,d\nu= 1$ it is clear that $\sup \vr \geq 1$ and $\inf \vr \leq 1$.
Together with the remark that any $\vr \in \La_{\kappa,\de}$ satisfies
$\sup \vr \leq [1+m\kappa \diam(M)^\alpha] \inf \vr$ this shows that
\begin{equation}\label{eq.supinf}
\frac{1}{R_1}
	\le \inf \vr 
	\le 1
	\le \sup \vr 
	\le R_1
\end{equation}
where $R_1=1+m\kappa \diam(M)^\alpha$.
Write $\vr_n=\tilde \cL_\phi^n(\vr)$.  First notice that $(\vr_n)_n$ is an equi-H\"older sequence 
since $|\vr_n(x)-\vr_n(y)| \leq \kappa \inf \vr \;d(x,y)^\al \leq \kappa d(x,y)^\al$ for all $d(x,y)<\de$ and all $n$,
which proves that all $\vr_n$ are $\kappa m$-H\"older continuous.

From the previous discussion we know that $\int \vr_n d\nu = 1$ for every $n$
and, consequently, the sequence $\vr_n$ is uniformly bounded from above and below. 
In fact, observe first that $\int \vr_k d\nu= \int \vr_l d\nu=1$ implies 
$\inf \frac{\vr_k}{\vr_l} \leq 1 \leq \sup \frac{\vr_k}{\vr_l}$. Therefore, from \eqref{eq.cont} we get
$$
e^{-\Delta \tau^n} 
	< \frac{\sup_{x \in M}\{\vr_k(x)/ \vr_l(x)\}}{\inf_{y \in M}\{\vr_k(y)/ \vr_l(y)\}} 
	= e^{\theta^+(\vr_k,\vr_l)}
	< e^{\Delta \tau^n} 
$$
for every $k$ and $l \geq n$. In consequence, 
\begin{equation}\label{eqqq}
e^{-\Delta \tau^n}  
	< \inf \frac{\vr_k}{\vr_l} 
	\leq 1 \leq \sup \frac{\vr_k}{\vr_l}
	< e^{\Delta \tau^n}
\end{equation}
and $(\vr_k)_k$ is a Cauchy sequence in the $C^0$-norm. In fact, 
\begin{equation}\label{eq.supconv}
\sup |\vr_k- \vr_l| 
	\leq \sup \left( |\vr_l| \left|\frac{\vr_k}{\vr_l}- 1\right| \right)
	\leq R_1 (e^{\Delta \tau^n} - 1)
	\leq 3 R_1 \Delta \tau^n
\end{equation}
for every $k, l \geq n$ and any $n \geq -\log (\Delta)/\log(\tau)$.
This yields that $(\vr_k)$ converges uniformly to some function $h$ in $\La_{\kappa,\de}$ satisfying
$\int h d\nu= 1$ and, consequently, $\kappa m$-H\"older continuous. It follows from a standard argument 
that $\mu= \int h \,d\nu$ is an $f$-invariant probability measure. Furthermore, the sequence $(\tilde \cL_\phi^n(\psi))_n$ converges to the same limit for any normalized function $\psi\in \La_{\kappa,\de}$. Indeed, 
if this was not the case then the same arguments used before are enough to conclude that the sequence
$$
\psi_n:= \begin{cases}\begin{matrix}
\vr_n, & \text{ if } n \text{ is odd } \\
\tilde \cL_\phi^n(\psi), & \text{otherwise} \\
\end{matrix}\end{cases}
$$
is Cauchy and, consequently, converges. This shows that the functions
$\tilde \cL_\phi^n(\vr)$ and $\tilde \cL_\phi^n(\psi)$ must have the same limit
and proves the uniqueness of the H\"older invariant density $h\in C^\alpha(M,\mathbb R)$
such that $\tilde \cL_\phi h =h$. By Theorem~B and Lemma~6.5 in \cite{VV10} we know that
equilibrium states coincide with invariant probability measures absolutely continuous with respect 
to $\nu$. Hence, $\mu=h\, \nu$ is an equilibrium state for $f$ with respect to $\phi$.
This finishes the proof of the proposition.
\end{proof}

Here we provide further information on the velocity of convergence to the invariant density in the space of H\"older continuous observables. More precisely,

\begin{corollary} \label{corfast}
Set $\vr \in \La_{\kappa,\de}$ be such that $\int \varphi \, d\nu=1$ and let $h$ denote the $\Theta_\kappa$-limit of 
$\vr_n= \tilde \cL^n_\phi(\vr)$. Then, $\vr_n$ converges exponentially fast to $h$ in the H\"older norm.
\end{corollary}

\begin{proof}
It follows from and \eqref{eqqq} and \eqref{eq.supconv} that $|\vr_n- h|_{\infty}  \leq  3 R_1 \Delta \tau^n$ and
\begin{equation}\label{eq.est1}
e^{-\Delta \tau^n}  
	\leq \inf \frac{\vr_n}{h} \leq 1 
	\leq \sup \frac{\vr_n}{h}
	\leq  e^{\Delta \tau^n}
\end{equation}
for every $n\in\mathbb N$. Now we claim that  $B_\kappa(h, \vr_n) \geq 1$. 
In fact this is immediate in the case that $\vr_n \equiv h$. Assume otherwise, by contradiction, and
notice that $B_\kappa(h, \vr_n) < 1$ implies $\varphi_n\neq h$.
Using \eqref{eq.est1}, there exists a point $z= z_n \in M$ such that $\vr_n(z) > h(z)$. Take $x_0$ such that
$\vr_n(x_0)- h(x_0)= \min \{\vr_n-h\}$. Therefore, if $0< d(w, x_0)< \delta$ we obtain that
$$
\frac{\vr_n(w)-\vr_n(x_0)}{d(w, x_0)^\alpha} \geq \frac{h(w)- h(x_0)}{d(w, x_0)^\alpha}.
$$
In consequence
$$
B_\kappa(h, \vr_n) 
	\geq \frac {\vr_n(z_n) -  (h(w)- h(x_0))/\kappa d(w, x_0)^\alpha}{h(z_n)
		 -  (\vr_n(w)- \vr_n(x_0))/\kappa d(w, x_0)^\alpha } 
	\geq 1.
$$
Analogously, one concludes that $A_\kappa(h, \vr_n) \leq 1$. Using the definition of $\Theta_\kappa$ and
the exponentially fast $\Theta_\kappa$-convergence of $\vr_n$ we get
$
e^{-\Delta \tau^n} 
	< A_\kappa(h, \vr_n) 
	\leq 1 \leq B_\kappa(h, \vr_n) 
	\leq e^{\Delta \tau^n}, \forall n \in \natural.
$
For notational simplicity, given $x\neq y$, set $H_{h}(x,y)=(h(x)- h(y))/\kappa d(x, y)^\alpha$ and $H_{\vr_n}$ be the corresponding expression for $\vr_n$. The previous estimates imply that
$e^{\Delta \tau^n} H_h(x,y) - H_{\vr_n}(x,y) \leq e^{\Delta \tau^n} \vr_n(z) - h(z)$. In particular 
$$
H_h(x,y)- H_{\vr_n} (x,y)
	< \vr_n(z) - h(z) + (e^{\Delta \tau^n} - 1)\cdot (\vr_n(z)- H_{\vr_n}(x, y))
	\leq 5 R_1 \Delta \tau^n
$$
for every large $n$. Since the other inequality follows from completely analogous computations 
one deduces that  $|h - \vr_n|_{\alpha, \delta} \leq 5 R_1 \Delta \tau^n$ for every large $n$. Therefore,
$|h - \vr_n|_{\alpha} \leq 5 m R_1 \Delta \tau^n$ which together with the previous estimate
$\|h - \vr_n\|_{\infty} \leq 3 R_1 \Delta \tau^n$ proves the corollary.
\end{proof}

The strict invariance of the cone $\La_{\kappa,\de}$ is now enough to obtain a spectral gap property
for the normalized operator $\tilde\cL_\phi=\la_\phi^{-1}\cL_\phi$. 

\begin{theorem}{(Spectral Gap)}\label{t.gap}
There exists $0< r_0< 1$ such that the operator $\tilde \cL_\phi$ acting on the space $C^\al(M,\mathbb R)$ 
admits a decomposition of its  spectrum given by $\Sigma= \{1\} \cup \Sigma_0$, where $\Sigma_0$ contained in a
ball $B(0, r_0)$.
\end{theorem}

\begin{proof}
Let $E_1$ be the one-dimensional eigenspace relative to the eingenvalue 1, and let
$E_0:= \{\vr \in C^\alpha(M,\mathbb R): \int \vr \,d\nu= 0\}$. Observe that $\int h\,d\nu=1$,  
the subspaces $E_0, E_1$ are $\tilde\cL_\phi$-invariant and $C^\alpha(M,\mathbb R)= E_1 \oplus E_0$: given 
$\varphi\in C^\alpha(M,\mathbb R)$ just write  $\varphi= \int \varphi\,d\nu . h+[\varphi- \int \varphi\,d\nu .  h]$.
Therefore, to obtain the spectral gap property it is enough to prove that $\tilde \cL_\phi^n |_{E_0}$ 
is a contraction for any large $n$. 

Take $\kappa \ge 1$ large such that $\La_{\kappa,\de}$ is preserved by $\tilde \cL_\phi$.
Pick $\varphi \in E_0$ with norm less or equal to $1$ and notice that $\vr+ 2 \in \La_{\kappa,\de}$
because $|\vr+2|_{\al,\de}=|\vr|_{\al,\de}\le 1$ and also $1 \le \kappa \inf |\vr+2|$.
Therefore $\tilde \cL_\phi^n (\varphi+2)$ converges to $ \int (\varphi+2) \,d\nu \cdot h=2 h$ and 
$$
\|\tilde \cL_\phi^n(\vr)\| 
	= \|\tilde \cL_\phi^n(\vr + 2)- \tilde\cL_\phi^n(2)\|  
	\leq   \|\tilde \cL_\phi^n(\vr+ 2)- 2 h\| +  \|\tilde \cL_\phi^n(2)- 2 h\| 
	\leq  20 K R_1 \Delta \tau^n,
$$
is exponentially contracted. This concludes the proof of the theorem.
\end{proof}

A first consequence of the spectral gap  is the following strong convergence.

\begin{corollary}
The equilibrium state $\mu$ coincides with the limit of the push-forwards $(f^j)_*\nu$ of the
conformal measure $\nu$.
\end{corollary}

\begin{proof}
First recall that $\cL^*\nu=\la\nu$. Thus, given any $\varphi\in C^0(M)$ it follows that
$
\int \varphi \; d (f^j_*\nu) = \int \varphi \circ f^ j \; d \nu
	= \int \varphi (\la^{-j} \cL^ j 1) \; d \nu
$
which converges to $\int \varphi h \; d \nu=\int \varphi \;d\mu$ as $j$ tends to infinity. 
Since $\varphi$ is arbitrary this proves that $\mu=\lim  f^j_*\nu$ as claimed.
\end{proof}


\subsubsection{Uniqueness of equilibrium states and exponential decay of correlations}

In this subsection we show that there is a unique equilibrium state for $f$ with respect to $\phi$
and derive good mixing properties.

\begin{theorem}\label{t.decay}
The equilibrium state $\mu=\mu_\phi$ has exponential decay of correlations for H\"older
observables: there exists $0<\tau<1$ such that for all $\varphi\in L^1(\nu), \psi\in C^{\alpha}(M)$
there is $K(\varphi,\psi)>0$ such that
\begin{equation*}
\left|\int (\varphi\circ f^n)\psi d\mu - \int \varphi d\mu\int \psi d\mu\right|\leq
K(\varphi,\psi)\cdot\tau^n,
	\quad \forall n \geq 1
\end{equation*}
\end{theorem}

\begin{proof}
First we write the correlation function
\begin{equation*}
C_{\varphi,\psi} (n):=\int (\varphi\circ f^n)\psi d\mu - \int \varphi d\mu\int \psi d\mu
	=\int (\varphi\circ f^n)\psi h d\nu - \int \varphi d\mu\int \psi d\mu.
\end{equation*}
It is no restriction to assume that $\int \psi d\mu = 1$. Then, using that $h$ is bounded away 
from zero and infinity we get
\begin{equation*}
\begin{split}
\left|\int (\varphi\circ f^n)\psi h d\nu - \int \varphi d\mu\int \psi d\mu\right| &=
\left|\int \varphi\left(\frac{\widetilde{\cL}_{\phi}^n (\psi h)}{h}-1\right)
d\mu\right|\\
&\leq \left\|\frac{\widetilde{\cL}_{\phi}^n(\psi h)}{h}-1\right\|_{0}\cdot\|\varphi\|_1
\end{split}
\end{equation*}
where $\|\varphi\|_1=\int |\varphi| d\mu$. If $\psi h\in\Lambda_{\kappa,\delta}$ for some sufficiently large $\kappa$
as in Theorem~\ref{t.cone.invariance} then it follows from \eqref{eq.supconv} that the first term in the right hand side
above satisfies
$$
\left\|\frac{\widetilde{\cL}_{\phi}^n(\psi h)}{h}-1\right\|_{0}
	\leq 2R_1\left\|\frac1h\right\|_0(e^{\Delta \tau^n}-1)
	\leq C\tau^n,
$$
for some positive constant $C$ and so
$
\left|\int (\varphi\circ f^n)\psi d\mu - \int \varphi d\mu\int \psi d\mu\right| 
	\leq K(\varphi,\psi)\tau^n.
$
In general write $\psi h=g$ where $g = g_B^+ - g_B^-$  and $g_B^\pm = \frac{1}{2} (|g| \pm g)+B$
for $B>0$ large so that $g_B^\pm\in \La_{\kappa,\de}$ and apply the latter estimates to $g_B^\pm$.
By linearity, the same estimate holds for $g$ for some constant 
 $K(\varphi,\psi)\geq K(\varphi,g_B^+)+K(\varphi,g_B^-)$.
This concludes the proof of the exponential decay of correlations.
\end{proof}

As a consequence we remove the topologically mixing assumption from \cite{VV10} and still deduce that there 
exists a unique equilibrium state and it is exact.

\begin{corollary}
The probability measure $\mu$ is exact and the unique equilibrium state for $f$ with respect to $\phi$.
\end{corollary}

\begin{proof}
Let $\vr \in L^1(\mu)$ be  such that $\vr= \vr_n \circ f^n$ for some measurable functions $\vr_n$.
Given any $\psi \in C^\alpha(M)$ it follows from the previous theorem that 
$$
\Big|\int (\vr - \int \vr d\mu)  \psi d\mu\Big|
	= \Big|\int (\vr_n \circ f^n) \psi d\mu - \int\vr d\mu \int \psi d\mu \Big| 
	\leq K(\vr_n, \psi) \tau^n,
$$
where the constant $K(\vr_n, \psi)$ depends only on the value of $\int \vr_n \,d\mu= \int \vr \,d\mu$
and $\|\psi\|_\alpha$. Hence  $K(\vr_n, \psi)$ does not depend on $n$ and, consequently, 
$\int (\vr- \int \vr d\mu) \psi d\mu= 0$, for all H\"older continuous $\psi$. The later implies that 
$\vr= \int \vr \, d\mu$ for $\mu$-almost every $x$, proving that $\mu$ is exact.
In consequence, $\mu = h \nu$ is an ergodic probability measure whose basin of attraction contains 
$\nu$-almost every point. Therefore the uniqueness of the equilibrium state follows from 
Theorem~\ref{thm.VV}. 
\end{proof}

\subsubsection{Central limit theorem}

Here we obtain a central limit theorem from the strong mixing properties.
Let $\cF$ be the Borel sigma-algebra of $M$ and $\cF_n:=f^{-n}(\cF)$ be
a non-increasing family of $\sigma$-algebras. Recall that a function
$\xi:M\to\real$ is $\cF_n$-measurable iff $\xi=\xi_n\circ f^n$ for some measurable $\xi_n$. 
Let $L^2(\cF_n)=\{\xi\in L^2(\mu):\xi \text{ is }\cF_n\text{-measurable }\}$ and
note that $L^2(\cF_n)\supset L^2(\cF_{n+1})$ for each $n\geq 0$. Given
$\varphi\in L^2(\mu)$, we denote by $\mathbb{E}(\varphi|\cF_n)$ the
$L^2$-orthogonal projection of $\varphi$ to $L^2(\cF_n)$. 
The strategy now is to apply a general result due to Gordin
by proving that the $L^2(\cF_n)$ components $\mathbb{E}(\varphi|\cF_n)$ of any observable $\varphi$ 
are summable.

\begin{lemma}\label{l.projections}
For every $\alpha$-H\"older continuous function $\varphi$ with $\int \varphi d\mu = 0$ there is $R_0=R_0(\varphi)$ 
such that $\|\mathbb{E}( \varphi |\cF_n)\|_2\leq R_0\tau^n$ for all $n\geq 0$.
\end{lemma}

\begin{proof} 
Observe that since $\|\psi\|_1\leq\|\psi\|_2$ and $\int \varphi d\mu = \int \varphi h d\nu = 0$ it follows
that
\begin{equation*}
\begin{split}
\|\mathbb{E}( \varphi |\cF_n)\|_2 &= \sup\Big\{\int\xi \varphi d\mu: \xi\in L^2(\cF_n),
\|\xi\|_2 = 1\Big\} \\
&= \sup\Big\{\int(\psi\circ f^n) \varphi \; d\mu: \psi\in L^2(\mu), \|\psi\|_2 = 1\Big\} 
\leq K( \varphi,\psi) \tau^n,
\end{split}
\end{equation*}
which proves the lemma.
\end{proof}

Now the central limit theorem in Corollary~\ref{c.CLT}  follows from the following 
abstract result due to Gordin (see e.g.~\cite{Vi97}).

\begin{theorem} 
Let $(M,\cF,\mu)$ be a probability space, $f:M\to M$ be a
measurable map such that $\mu$ is $f$-invariant and ergodic. Consider $\varphi\in
L^2(\mu)$ such that $\int \varphi d\mu = 0$ and denote by $\cF_n$ the non-increasing
sequence of sigma-algebras $\cF_n=f^{-n}(\cF), n\geq 0$. 
If $\sum\limits_{n=0}^{\infty}\|\mathbb{E} (\varphi |\cF_n)\|_2 <\infty$
then $\sigma_\varphi$ is finite, and $\sigma_\varphi = 0$ iff $\varphi=u\circ f - u$ for some $u\in
L^2 (\mu)$. Moreover, if $\sigma_\varphi>0$ then for any interval $A\subset\real$
\begin{equation*}
\mu\Big(x\in M: \frac{1}{\sqrt{n}}\sum\limits_{j=0}^{n-1} \left(\varphi(f^j(x))-\int \varphi d\mu \right)\in A\Big)
	\to \frac{1}{\sigma_\varphi \sqrt{2\pi}} \int_A e^{-\frac{t^2}{2\sigma^2}} dt,
\end{equation*}
as $n$ tends to infinity.
\end{theorem}

\subsubsection{Uniform continuity of the densities for the equilibrium states}

Here we shall prove the first stability result for the equilibrium state: the density of the equilibrium state with 
respect to the corresponding  conformal measure vary continuously in the $C^0$-norm. This is not immediate since
the Ruelle-Perron-Frobenius operator in general does not vary continuously with the dynamical system in the space 
of H\"older continuous observables as discussed in Example~\ref{ex.transfer}. 
Nevertheless, we could get the continuity of the density function which is the main result of this section.

\begin{proposition}\label{prop:C0cont}
Let $\cF$ be a family of local homeomorphisms and $\cW$ be a family of potentials satisfying (H1), (H2) and
(P) with uniform constants. Then the topological pressure 
$\cF\times \cW \ni (f,\phi) \mapsto \log \la_{f,\phi}= P_{\text{top}}(f,\phi)$ and  the density function
\[
\begin{array}{ccc}
\mathcal F\times \cW & \to & (C^\al(M,\mathbb R),\|\cdot \|_0) \\
(f,\phi) & \mapsto & \frac{d\mu_{f,\phi}}{d\nu_{f,\phi}}
\end{array}
\]
are continuous.
\end{proposition}

\begin{proof}
Recall that Proposition~\ref{p.pressure} implies that $P_{\text{top}}(f,\phi)=\log \la_{f,\phi}$ where $\la_{f,\phi}$
is the spectral radius of the operator $\cL_{f,\phi}$. Moreover, it follows from the proof of Corollary \ref{corfast} that
for any $\vr \in \La_{\kappa,\de}$ satisfying $\int \varphi \, d\nu=1$ one has in particular
\begin{equation}\label{eq.C0contt}
\left\|\la^{-n}_{f,\phi}\cL_{f,\phi}^{n}\varphi- \frac{d\mu_{f,\phi}}{d\nu_{f,\phi}}\right\|_{0}  
	\leq  3 R_1 \Delta \tau^n.
\end{equation}
for all $n$. Notice the previous reasoning applies to $\varphi\equiv 1\in\La_{\kappa,\de}$. Moreover, since
the spectral gap property estimates depend only on the constants $L,\si$ and $\deg(f)$ it follows that all transfer
operators $\cL_{\tilde f,\tilde \phi}$ preserve the cone $\La_{\kappa,\de}$ for all pairs $(\tilde f,\tilde\phi)$ and that
the constants $R_1$ and $\De$ can be taken uniform in a small neighborhood $\cU$ of $(f,\phi)$. Furthermore,
one has that $\int \la_{f,\phi}^{-1}\cL_{f,\phi}\; d\nu_{f,\phi}=1$ and so the convergence
$$
\lim_{n \to +\infty} \frac 1 n \log \|{\tilde {\cL^n}}_{\tilde f,\phi}(1) \|_0
	= \lim_{n \to +\infty} \frac 1 n \log\left\| {\la_{\tilde f, \phi}}^{-n} {\cL^n}_{\tilde f , \phi}  (1) \right\|_0 
	= 0
$$
given by Proposition \ref{p.densidade} and Corollary \ref{corfast} can be taken uniform in $\cU$. This is the key
ingredient to obtain the continuity of the topological pressure and density function. Indeed, let $\vep>0$ be fixed
and take $n_0\in \mathbb N$ such that 
$$
\Big |\frac{1}{n_0} \log \|{\cL^{n_0}}_{\tilde f, \phi}(1)\|_0 - \log(\la_{\tilde f, \phi})\Big |
	< \frac\epsilon3.
$$ 
for all $\tilde f  \in \cU$.
Moreover,  using $P_{\text{top}}(f,\phi)=\log \la_{f,\phi}$ by triangular inequality we get
\begin{align*}
\Big | \Ptop(f,\phi) - \Ptop(\tilde f,\phi)\Big |
	& \leq \Big |\frac{1}{n_0} \log \|{\cL^{n_0}}_{\tilde f, \phi}(1)\|_0 - \log(\la_{\tilde f, \phi})\Big | \\
	&+ \Big |\frac{1}{n_0} \log \|{\cL^{n_0}}_{ f, \phi}(1)\|_0 - \log(\la_{ f, \phi})\Big | \\
	& + \Big |\frac{1}{n_0} \log \|{\cL^{n_0}}_{ f, \phi}(1)\|_0 - \frac{1}{n_0} \log \|{\cL^{n_0}}_{ \tilde f, \phi}(1)\|_0 \Big |.
\end{align*}
Now, it is not hard to check that, for $n_0$ fixed, the function $\cU \to C^0(M,\mathbb R)$ given by
$$
\tilde f \mapsto \cL^{n_0}_{\tilde f,\phi} 1
	=\sum_{\tilde f^{n_0}(y)=x} e^{S_{n_0}\phi(y)}
$$
is continuous. Consequently, there exists a neighborhood $\mathcal V\subset \cU$ of $f$ such that 
$|\frac{1}{n_0} \log \|{\cL^{n_0}}_{ f, \phi}(1)\|_0 - \frac{1}{n_0} \log \|{\cL^{n_0}}_{ \tilde f, \phi}(1)\|_0|<\vep/3$
for every $\tilde f\in \mathcal V$. Altogether this proves that $| \Ptop(f,\phi) - \Ptop(\tilde f,\phi)\Big |<\vep$
for all $\tilde f \in \mathcal V$. Since $\vep$ was chosen arbitrary we obtain that both the leading eigenvalue 
and topological pressure functions vary continuously with the dynamics $f$.
Finally, by equation~\eqref{eq.C0contt} above applied to $\varphi\equiv 1$ and triangular inequality 
we obtain that
\begin{align*}
\left\|\frac{d\mu_{\tilde f,\phi}}{d\nu_{\tilde f,\phi}} -  \frac{d\mu_{f,\phi}}{d\nu_{f,\phi}} \right\|_{0} 
	& \leq 6 R_1 \Delta \tau^n 
	+ \left\|\la^{-n}_{\tilde f,\phi}\cL_{\tilde f,\phi}^{n}1- \la^{-n}_{f,\phi}\cL_{f,\phi}^{n}1  \right\|_{0}
\end{align*}
for all $n$. Hence, proceeding as before one can make the right hand side above as close to zero as 
desired provided that $\tilde f$ is sufficiently close to $f$. This proves the continuity of the density function and
finishes the proof of the proposition. 
\end{proof}

We will finish this section with some comments on the non-continuous dependence of the Ruelle-Perron-Frobenius operators, acting on the space of H\"older continuous observables, as a function of the dynamics $f$. 

\begin{remark}\label{rmk.continuity}
Notice first that H\"older continuous observables are Lipschitz continuous 
with respect to the metric $d(\cdot, \cdot)^\al$. Hence, for simplicity we provide below an example of 
discontinuity of the Ruelle-Perron-Frobenius operator $\cL_{f} \colon Lip(M) \to Lip(M)$
with respect to the dynamics $f$, where $Lip(M)$ are the space of continuous observables such that
$$
Lip(f):=\sup_{n\neq y} \frac{|f(x)-f(y)|}{d(x,y)}
	<\infty.
$$
\end{remark}

\begin{example}\label{ex.transfer}
The key idea of the following surprisingly simple example of 
discontinuity of Ruelle-Perron-Frobenius operator with respect to the dynamics
is that the operator of composition $\vr \to \vr \circ g$ acting in the
space of Lipschitz functions does not vary continuously with $g$. 
Consider the expanding dynamics $f_n$ on the circle $S^1 \simeq \re/[-1/2, 1/2)$
given by that $f_n(x)= 2(x+ \frac1{10n}) (\!\!\mod 1)$. Obviously, $f_n$  
converges to $f$, $f(x)= 2 x (\!\!\mod 1)$ in the $C^\infty$-topology.
Now, take a periodic Lipschitz function $\varphi$ in the circle   
such that $\varphi(x)= |x|$ say, for $|x| \leq 1/8$ and
$\varphi(x)= 0$  for $1/2 \geq |x|\geq 1/5$.
Just take the potential $\phi \equiv 0$ and
write $\cL_n$, $\cL$ for the Perron-Frob\"enius operators 
corresponding to $f_n ,f$ respectively.
Therefore,  taking $0< x_n< y_n< 1/10n$, we obtain that
\begin{align*}
\text{Lip}((\cL_n- \cL)(\vr)) 
	& \geq  \frac{|\cL_n(\varphi)(y_n)- \cL_n(\varphi)(x_n) +\cL(\varphi)(x_n)- \cL(\varphi)(y_n)|}{y_n- x_n} \\
	& =  \frac{| |y_n/2 - 1/10n|- |x_n/2- 1/10n|+ |x_n/2|- |y_n/2||}{y_n- x_n} \\
	& = \frac{ |-y_n - x_n|}{y_n- x_n}= 1 = Lip(\varphi).
\end{align*} 
Thus $\cL_n:\text{Lip}(S^1, \re) \to \text{Lip}(S^1, \re)$ 
does not converge to $\cL$ even in the strong operator topology.
In particular, $\cL_n$ does not converge to $\cL$ in the norm topology.
 \end{example}

\section{Ruelle-Perron-Frobenius in $C^r(M,\mathbb R)$: \\ spectral gap and strong stability results}\label{s.gapCr}

Throughout this section we assume that  $f$ is a $C^r$-local diffeomorphism ($r\ge 1$)
and the potential $\phi$ belongs to $C^r(M)$. Here we restrict the analysis of the transfer operator to the space 
of smooth observables. 

\subsection{Spectral gap for the transfer operator in $C^{r}(M,\mathbb R)$}\label{s.s.conesCr}

Here we shall assume that $f$ is a $C^r$ ($r\ge 1$) local diffeomorphism on a compact manifold $M$ 
satisfying (H1) and (H2) and $\phi\in C^{r}(M,\mathbb R)$ satisfies (P'). In fact,  we require $L\ge 1$ to be 
close to $1$ such that
\begin{equation}\label{eq.updated}
\Xi_r:=e^{\vep_\phi}\frac{q L^r +(\deg(f)-q) \si^{-1}}{\deg(f)}<1,
\end{equation}
which we use as counterpart of \eqref{eq. relation potential} in this differentiable setting.
We prove that the transfer operator 
$\cL_\phi : C^{r}(M,\mathbb R) \to C^{r}(M,\mathbb R)$ has a spectral gap. 
The strategy is to show $\cL_\phi$-invariance of the cones of smooth observables 
$$
\Lambda^{r}_{\kappa}
	=\Big\{\varphi\in C^{r}(M,\mathbb R) : \varphi>0 \text{ and } \frac{\|D^s\varphi\|_0}{\inf \varphi} 
		\leq \frac{\kappa}{c^{(r)}_s} \quad\text{for} \;s=1\dots r  \Big\},
$$
for some constants $c^{(r)}_s$ with $s=1\dots r$ defined recursively using the corresponding constants 
for the cones corresponding to smaller differentiability. The choice of the constants $c^{(r)}_s$ are made in
order to guarantee that observables in $\Lambda^{r}_{\kappa}$ associated to large $\kappa$ belong 
to some cones $\Lambda^{r-i}_{\kappa_i}$ for some large constants $(k_i)_{i=1\dots r-1}$ where 
the Ruelle-Perron-Frobenius operator acts as a contraction in the projective metric.
The precise construction is described in what follows. For $r=1$ just consider
$$
\Lambda^{1}_{\kappa}
	=\Big\{\varphi\in C^{r}(M,\mathbb R) : \varphi>0 \text{ and } \frac{\|D \varphi\|_0}{\inf \varphi} 
		\leq \kappa   \Big\},
$$
which corresponds to the previous cone with $c^{(1)}_1=1$. For $r=2$ we obtain that the cone
$\Lambda^{2}_{\kappa}$ can be written as
$$
\Lambda^{2}_{\kappa}
	=\Big\{\varphi\in C^{2}(M,\mathbb R) : \varphi>0, \; \frac{\|D \varphi\|_0}{\inf \varphi} 
		\leq \frac{(1-\Xi_2) \kappa}{2 e^{\vep_\phi} \max_x \|D^2 f^{-1}(x)\|}  \text{and }  \frac{\|D^2 \varphi\|_0}{\inf \varphi} 
		\leq \kappa  \Big\}.
$$
with $c^{(2)}_2=1$ and $c^{(2)}_1=2 (1-\Xi_2)^{-1} e^{\vep_\phi} \max_x \|D^2 f^{-1}(x)\|$. Assuming that 
we have defined the positive constants $(c^{(r-1)}_s)_{s=1\dots r-1}$ associated to the cones 
$\Lambda^{r-1}_{\kappa}$ of $C^{r-1}$ observables we define the constants  $c^{(r)}_s$ as follows. Set 
\[
\begin{cases}
c^{(r)}_r=1;  \\
c^{(r)}_{r-1}= r! (1-\Xi_r)^{-1} e^{\vep_\phi} \max_{1\le j\le r-1}\max_x \|D^{r-1}f^{-1}(x)\|^j ; \\
c^{(r)}_{r-t} = c^{(r)}_{r-t+1} \cdot c^{(r-1)}_{r-t}, \text{ for } t=2\dots r-1
\end{cases}
\]
Roughly, the choice of $c^{(r)}_{r-1}$ is made in order to guarantee that at most $r!$ terms arising in the
computation of  higher order derivatives of the observable $\cL_\phi\varphi$ are dominated by the term involving
$D^r\varphi$,  while the recursive choice of the constants $c^{(r)}_{s}$ with $s<r$ guarantees that the cones 
corresponding to smaller differentiability are contracted.  Hence, our main result in this section is as follows.

\begin{theorem}\label{t.cone.invarianceCr}
There exists a positive constant $\vep'_\phi>0$ (depending only on $f$ and $r$) such that if $\phi$ satisfies condition
(P') given by
\begin{equation*}
 \sup\phi-\inf\phi<\vep_\phi \quad  \text{and} \quad \max_{s\le r}\|D^s \phi\|_0 <\vep'_\phi 
\end{equation*}
then there are $\kappa_0>0$ and $0< \hat \la< 1$ such that $\cL_\phi(\Lambda^{r}_{\kappa}) \subset \Lambda^{r}_{\hat\la \kappa,\delta} $ for every $\kappa\ge \kappa_0$.
\end{theorem}

\begin{proof}
We shall prove the theorem recursively on the differentiability $r$. First set $r=1$ and consider 
$\varphi\in \Lambda^{1}_{ \kappa}$ for $\kappa>0$ large. Given $x\in M$ let
$(x_j)_{j}$ denote the set of preimages by $f$ of the point $x$ and 
denote by $f_i^{-1}$ corresponding the local inverse branch for the function $f$ in a neighborhood of $x$ with $f_i(x_i)=x$.  It is not hard to check that $|\cL_\phi \varphi(x)|\leq e^{\vep_\phi} \inf |\cL_\phi \varphi |$
for every $x\in M$.
Moreover, 
\begin{align}\label{eq.chainC1}
D(\cL_\phi \varphi)(x)
	& = \sum_{j= 1}^{\deg(f)} e^{\phi(x_j)} D\varphi(x_j) Df_j^{-1}(x)  
	+ \sum_{j= 1}^{\deg(f)}   \varphi(x_j) e^{\phi(x_j)} D\phi (x_j) Df_j^{-1}(x)
\end{align}
and, consequently, $\| D(\cL_\phi \varphi)(x))\|$ is bounded by
\begin{align*}
\sum_{j= 1}^{\deg(f)} | e^{\phi(x_j)}|   \left\| D\varphi(x_j) Df_j^{-1}(x)  \right\| 
	  + \sum_{j= 1}^{\deg(f)}   |\varphi(x_j) e^{\phi(x_j)}|  \left\| D\phi (x_j) Df_j^{-1}(x) \right\|. 
\end{align*}
By our assumptions  (H1) and (H2) it follows that the isomorphism $\|D f_j^{-1}(\cdot)\| \leq L$ 
for $j=1\dots q$ and is indeed a contraction for $j>q$. Thus, using (P') and that 
$\sup\varphi \leq \inf\varphi+ \|D\varphi\|_0 \diam (M)$ we get
\begin{align*}
\frac{\| D(\cL_\phi \varphi)(x))\|}{\inf  |\cL_\phi \varphi|}
	& \leq \frac{\sum_{j= 1}^{q} L | e^{\phi(x_j)}|   \left\| D\varphi(x_j) \right\|_0   
	+ \sum_{j>q } \sigma^{-1} | e^{\phi(x_j)}|   \left\| D\varphi(x_j)  \right\|_0}
	{ \deg(f) e^{\inf \phi} \inf \varphi }\\
	& + \frac{\sum_{j= 1}^{q} L |\varphi(x_j) e^{\phi(x_j)}|   \left\| D\phi(x_j) \right\|_0   
	+ \sum_{j>q} \sigma^{-1} |\varphi(x_j) e^{\phi(x_j)}|   \left\| D\phi(x_j)  \right\|_0}
	{ \deg(f) e^{\inf \phi} \inf \varphi }   \\
	& \leq e^{\vep_\phi} \;  \frac{q L + \si^{-1} (\deg(f)-q) }{\deg(f)} \; \frac{\| D\varphi\|_0}{\inf \varphi}  \\
	&  + e^{\vep_\phi} \| D\phi(x_j)  \|_0  \frac{\sup\varphi}{\inf \varphi} \; \frac{q L + \si^{-1} (\deg(f)-q) }{\deg(f)} \\
	& \leq \Xi_1 \cdot \kappa
	+ \Xi_1\, \vep'_\phi \; (1 + \|D\varphi\|_0\diam(M))
\end{align*}
which can be taken smaller than $\tilde \la \, \kappa $, for some constant $0<\tilde \la<1$ by our choice
of $\vep_\phi$ in \eqref{eq.updated} provided that $\vep'_\phi$ is sufficiently small. In consequence we 
obtain that $\| D(\cL_\phi \varphi)\|_0 \leq \tilde \la \kappa \, \inf |\cL_\phi \varphi|$, which proves
the theorem in the case that $r=1$.

We now consider the case $r=2$. Fix $\kappa>0$ and consider $\varphi\in \Lambda^{2}_{ \kappa}$.
Differentiating \eqref{eq.chainC1} by means of the chain rule we obtain sums involving the seven terms 
\begin{align*}
& D^2\phi (x_j) [D f_j^{-1}(x)]^2 e^{\phi(x_j)} \varphi(x_j)\\
& D\phi(x_j) D^2 f_j^{-1}(x) e^{\phi(x_j)} \varphi(x_j)\\
& D\phi(x_j) D f_j^{-1}(x) e^{\phi(x_j)} D\varphi(x_j) Df_j^{-1}(x)\\
& D\phi(x_j) D f_j^{-1}(x) e^{\phi(x_j)} D\phi(x_j) Df_j^{-1}(x)\\
& e^{\phi(x_j)} D\phi(x_j) Df_j^{-1}(x) D\varphi(x_j) Df_j^{-1}(x) \\
& e^{\phi(x_j)} D^2\varphi (x_j) [Df_j^{-1}(x)]^2\\
& e^{\phi(x_j)} D\varphi (x_j) D^2f^{-1}_j(x).
\end{align*}
Our assumption $\max_{s\le r}\|D^s \phi\|_0 <\vep'_\phi$ with $\vep'_\phi>0$ small implies that all  but  the last two 
previous terms can be taken neglectable.
In consequence, proceeding as before we conclude that there exists a uniform constant $C>0$ depending 
only on $f$ such that 
\begin{align*}
\frac{\| D^2(\cL_\phi \varphi)(x))\|}{\inf |\cL_\phi \varphi |}
	& \leq C \vep'_\phi
		+ e^{\vep_\phi} \;  \frac{\|D^2\varphi\|_0}{\inf \varphi}  \; 
		\frac{q L^2 + \si^{-2} (\deg(f)-q) }{\deg(f)}  \\ 
	&	+ e^{\vep_\phi} \;  \frac{\|D\varphi\|_0}{\inf \varphi} \, \max_{x\in M} \|D^2 f^{-1}(x)\| \\
	& \leq C \vep'_\phi + \Xi_2 \; \kappa 
		+ \frac{1-\Xi_2}{2}\; \kappa
\end{align*}
which can be again taken smaller than $\tilde \la \, \kappa $, for some constant $0<\tilde \la<1$, provided that
$\kappa$ is large enough and $\vep'_\phi$ is small. This estimate, which involves the information on the smaller
derivaties, proves the strict invariance of the cone $\Lambda^2_\kappa$ under the operator $\cL_\phi$.
The complete statement in the theorem follows by completely analogous computations for the $s$-derivatives of 
$\cL_\phi g$, with  $2< s\le r$. In fact,  the remaining of the proof can be obtained recursively for $\ell+1$ using 
previous information or $s\in\{1,2, \dots,  \ell\}$ by analogous computations of higher order derivatives using the 
chain rule  and estimating dominating terms as above. In fact the number of terms associated containing the
derivatives $D^s\varphi$, $s=1\dots r$ are clearly less than $r!$ and, by definition of the cones, each of such 
terms is bounded by $(1-\Xi_r)/r!$. Then if $\vep'_\phi$ is small proceeding as above we get that $\Lambda^r_\kappa$ is strictly preserved by $\cL_\phi$, 
which proves the theorem.
\end{proof}

Again we use that the smaller cone has finite diameter in the projective metric in the
case of the cones for differentiable observables, whose proof can be simply adapted from the one of
Proposition~\ref{propdiam}. For that reason we shall omit its proof.

\begin{proposition}
Given $0< \hat\la< 1$,  the cone $\Lambda^{r}_{\hat\la \kappa,\delta}$ has finite $\Lambda^r_{\kappa}$-diameter.
\end{proposition}

In the next subsection we will use Birkhoff's theorem to deduce good spectral properties for the 
Ruelle-Perron-Frobenius operator.

\subsection{Strong stability properties}

Here we establish the statistical and spectral stochastic stability results. The discussion in 
Remark~\ref{rmk.continuity} shows that this property was far from being immediate. 
In the space of $C^r$ observables ($r \geq 1$) the Perron-Frobenius vary continuously with 
the dynamics in the strong (pointwise) operator topology. However, it can also be shown that in general such operators do not vary continuously in norm in the space of bounded  linear operators.
In fact, the stability results presented here will follow from the careful study of the spectral properties
of the transfer operators and will be consequence  of the uniformity of the gap spectral for all
close dynamical systems and potentials.

Throughout this subsection let $\cF^r$ be a family of $C^r$, $r \geq 1$ local diffeomorphisms
and let $\mathcal W^r$ be some family of $C^r$ potentials satisfying (H1), (H2) and (P') with uniform constants. 
Let $B(C^r(M,\mathbb R), C^r(M,\mathbb R))$ denote the space of bounded linear operators on $C^r(M,\mathbb R)$ endowed with the strong operator topology.

\begin{proposition}\label{prox}
The Ruelle-Perron-Frobenius operator function
\[
\begin{array}{ccc}
\cF^r \times \cW^r  &  \to & B(C^r(M,\mathbb R), C^r(M,\mathbb R)) \\
   (f, \phi) & \mapsto & \cL_{f,\phi}
\end{array}
\]
is continuous in the $C^r$-topology.
\end{proposition}

\begin{proof}
Let  $(f,\phi), (\tilde f,\tilde \phi) \in \cF^r\times \cW^r$ be given. Then for any fixed $\vr \in C^r(M, \re)$ and $x\in M$
we get that 
$$
\|\cL_{\tilde f, \tilde \phi}  (\vr)  (x) -   \cL_{\tilde f, \tilde \phi}  (\vr)  (x)\| \leq 
\sum_{j = 1}^{\deg(f)} \|  \vr (\tilde f_i^{-1} ) (x)  \cdot e^{\tilde \phi( \tilde f_i^{-1} (x) ) }   
	-  \vr ( f_i^{-1} ) (x)  \cdot e^{ \phi( f_i^{-1} (x) ) }\|
$$
where, as before, $f_i^{-1}$ denote the inverse branches of $f$ at $x$. Moreover, the right hand
side above goes to zero independently of $x$ as  $(\tilde f,\tilde \phi)$ converges to $(f,\phi)$ in 
the $C^1$-topology. Furthermore, $\| D  \cL_{\tilde f, \tilde \phi}  (\vr)  (x) - D  \cL_{\tilde f, \tilde \phi}  (\vr)  (x)   \|$
is bounded by
\begin{align*}
\sum_{j = 1}^{\deg(f)} \| D( \vr (\tilde f_i^{-1} ) (x)  \cdot e^{\tilde \phi( \tilde f_i^{-1} (x) ) })   
	- D( \vr ( f_i^{-1} ) (x)  \cdot e^{ \phi( f_i^{-1} (x) ) })\|,
\end{align*}
which also converges uniformly  to zero by standard triangular argument as the element 
provided that $(\tilde f,\tilde \phi)$ converge to $(f,\phi)$. We note that analogous computations hold for
higher order derivatives which lead to the statement of the proposition.
\end{proof}

Now we deduce our functional analysis approach to deduce the important continuity of the topological 
pressure, a fact unknown in \cite{VV10}.

\begin{proposition} \label{propre}
The topological pressure function $\cF^r \times \cW^r \ni (f, \phi)  \to P_{\text{top}}(f,\phi)$ is
continuous in the $C^r$-topology, for $r \geq 1$. Moreover, the densities $\frac{d\mu_{f, \phi}}{d\nu_{f, \phi}}$ 
vary continuously with respect to $(f, \phi) \in \cF^r \times \cW^r$.
\end{proposition}

\begin{proof}
This proof goes along the same lines of the proof of Proposition~\ref{prop:C0cont}. For that reason
we will prove the result by focusing on the main differences. First notice that $\la_{f,\phi}$
is the leading eigenvalue and spectral radius of the operator $\cL_{f,\phi}$ acting in any space of 
the Banach spaces $C^r$ with $r \geq 0$. Now, using once more that all transfer operators associated
to all $(\hat f, \hat \phi)$ in some neighborhood $\cU$ of $(f, \phi)$ preserve the same cone of functions 
we obtained, following Proposition~\ref{p.densidade} and Corollary~\ref{corfast}, that
the limit 
$$
\lim_{n \to +\infty} \frac 1 n \log \|{\tilde {\cL^n}}_{\hat f,\hat \phi}(1) \|_r
	= \lim_{n \to +\infty} \frac 1 n \log\big \| \la_{f,\phi}^{-n} \, {\cL^n}_{\hat f ,\hat \phi} (1) \big \|_r 
	= 0
$$
is uniform for all $(\hat f, \hat \phi)$ in some neighborhood $\cU$ of $(f, \phi)$. Therefore, by 
standard triangular inequality together with the continuity of the transfer operators in the $C^r$-strong
topology it follows that given $\epsilon > 0$ there exists $n_0$ such that 
\begin{align*}
\Big | \Ptop(f,\phi) - \Ptop(\hat f,\hat\phi)\Big |
	& \leq \Big |\frac{1}{n_0} \log \|{\cL^{n_0}}_{\hat f, \hat\phi}(1)\|_r - \log(\la_{\hat f, \hat\phi})\Big | \\
	&+ \Big |\frac{1}{n_0} \log \|{\cL^{n_0}}_{ f, \phi}(1)\|_r - \log(\la_{ f, \phi})\Big | \\
	& + \Big |\frac{1}{n_0} \log \|{\cL^{n_0}}_{ f, \phi}(1)\|_r - \frac{1}{n_0} \log \|{\cL^{n_0}}_{ \hat f, \hat\phi}(1)\|_r 
	\Big |< \epsilon
\end{align*}
as $(\hat f,\hat\phi)$ converges to $(f,\phi)$. This argument shows that the leading
eigenvalue, thus the topological pressure, vary continuously. Proceeding as in the proof of Proposition~\ref{p.densidade}, noticing that $C^r(M,\mathbb R)\subset C^0(M,\mathbb R)$ and 
$\int \la^{-1}_{f,\phi}\cL_{f,\phi}\;d\nu=1$, one obtains from the contraction on the projective metric 
$\Theta_\kappa$that
\begin{align*}
\left\|\frac{d\mu_{\hat f,\hat\phi}}{d\nu_{\hat f,\hat \phi}} -  \frac{d\mu_{f,\phi}}{d\nu_{f,\phi}} \right\|_{r} 
	& \leq 6 R_1 \Delta \tau^n 
	+ \left\|\la^{-n}_{\hat f,\hat \phi}\cL_{\hat f,\hat \phi}^{n}1- \la^{-n}_{f,\phi}\cL_{f,\phi}^{n}1  \right\|_{r}
\end{align*}
where $R_1$ is a uniform upper bound for the $C^0$-norm of the iterates $\la^{-n}_{f,\phi}\cL^n_{f,\phi} 1$
in a neighborhood of $(f,\phi)$, the constant $\De$ is the diameter of the cone $\La_{\kappa}^r$, and $0<\tau<1$. 
Using the continuity of the transfer operators given in Proposition~\ref{prox} then for any fixed $n$ the last 
expression in the right hand side above can be made arbitrarily small provided that $(\hat f,\hat\phi)$
is sufficiently close to $(f,\phi)$. This proves the continuity of the density function and finishes the proof
of the proposition. 
\end{proof}

We are now in a position to prove that the equilibrium states are strongly stable under
deterministic perturbations.

\begin{proof}[Proof of Theorem~\ref{Thm.Statistical}]
The continuity of topological pressure given by Proposition~\ref{propre} together with
Theorem D in \cite{VV10} that the conformal measures $\nu_{f_n, \phi_n}$ converge to $\nu_{f, \phi}$ in the
weak$^*$ topology as $(f_n, \phi_n)$ goes to $(f, \phi)$ in the $C^r$ topology. Now, using that 
$\frac{d\mu_{f,\phi}}{d\nu_{f,\phi}}$ is $C^r$ and  also varies continuously in the $C^r$ norm with 
$(f, \phi) \in\cF^r \times \cW^r$ it follows that the equilibrium state $\mu_{f, \phi}$ also varies continuously in the
weak$^*$ topology, which completes the proof of the theorem.
\end{proof}

Finally we derive the strong stochastic stability of the spectra.
Consider any family $\theta_\vep$, $0< \vep \le 1$  of probability measures in $\cF^r\times \cW^r$ such that
its support $\supp\theta_\vep$ is contained in a neighborhood $V_\vep(f,\phi)$ of $(f,\phi)$ depending 
monotonically on $\vep$ and satisfying
$
\bigcap_{0<\vep\le 1} V_\vep(f,\phi) =\{(f,\phi)\}.
$
We refer to $(\theta_\vep)_\vep$ as an arbitrary random perturbation of $(f,\phi)\in \cF^r\times \cW^r$.  
We first prove that the stochastic transfer operator $\cL_\epsilon:C^r(M,\mathbb R) \to C^r(M,\mathbb R)$ given by 
\begin{equation}\label{eq.stoRPF}
\cL_\epsilon( \vr)= \int \cL_{\tilde f, \tilde \phi} \vr \; d\Theta_{\epsilon}(\tilde f,\tilde \phi)
\end{equation}
is well defined and preserves a cone of $C^r$-observables.

\begin{lemma}
The stochastic transfer operator operator $\cL_\epsilon$ defined in \eqref{eq.stoRPF}
is well defined. Moreover, there exists $0< \hat \la < 1$  so that  
$\cL_\vep(\La^r_\kappa)\subset \La^r_{\hat \la \kappa}$ for every small $\vep$ and every 
large $\kappa$.
\end{lemma}

\begin{proof}
First we prove that the stochastic transfer operator $\cL_\vep$ is well defined. Given any fixed 
$\varphi\in C^r(M,\mathbb R)$ it follows that $\cL_{\tilde f,\tilde \phi} (\varphi)$ is $C^r$ for all
$(\tilde f,\tilde \phi)\in \cF^r\times \cW^r$. Moreover, since the constants are taken uniform in
the family $\cF^r$ and $\cW^r$ then it is a consequence of Lebesgue dominated convergence theorem
that $\cL_\vep(\varphi)$ is also $C^r$. This proves the first claim in the lemma.

On the other hand, by construction we obtain $0<\hat\la<1$ and $\kappa$ large so that 
$\cL_{\tilde f,\tilde \phi}(\La^r_\kappa)\subset \La^r_{\hat \la \kappa}$ for every $(\tilde f,\tilde\phi)$ in 
a neighborhood of $(f,\phi)$. In particular, if $\vep$ is small then this property holds in $V_\vep(f,\phi)$ and, consequently, $\cL_{\vep}(\La^r_\kappa) \subset \La^r_{\hat \la \kappa}$. This proves the second statement 
finishes the proof of the lemma.
\end{proof}

We finish our section by proving our spectral stochastic stability result. 

\begin{proof}[Proof of Theorem~\ref{thm.spectral.stability}]
Let $(f,\phi)\in \cF^r\times \cW^r$ be fixed. By Proposition~\ref{prox} the transfer operators $\cL_{\tilde f,\tilde \phi}$ acting on the space $C^r(M, \re)$ vary continuously with $(\tilde f,\tilde \phi)\in \cF^r\times \cW^r$ in the strong 
operator topology.

Recall also that the dominant eigenvalue for $\cL_{f,\phi}$ equals to the spectral radius and has multiplicity one
and that both the leading eigenvalue and corresponding eigenspace vary continuously.  Moreover, since 
all transfer operators $\cL_{\tilde f,\tilde \phi}$ preserve the same cone $\La^r_\kappa$ for all 
$(\tilde f,\tilde \phi)$ in a small neighborhood of $(f,\phi)$ then it follows from the last proposition that
the same property holds for $\cL_\vep$ with $\vep$ small.  
Proceeding as in the later sections we get that $\cL_\vep$ has a spectral gap for every small $\vep$.
In particular, there exists a unique eigenvalue $\la_\vep$, which coincides with the spectral radius of $\cL_\vep$,
and the eigenspace associated to $\la_\vep$ is one-dimensional.

We claim that the spectral radius $\la_\vep$ of $\cL_\vep$ varies continuously for all small $\vep$ and that
it converges to $\la_{f,\phi}$ whenever $\vep$ tends to zero. If $\vep>0$ is small we have that
all operators $\la_\vep^{-1} \cL_\vep$ preserve the same cone of functions $\La^r_\kappa$. Moreover,
there exists a conformal measure $\nu_\vep$, that is, such that $\cL_\vep^*\nu_\vep=\la_\vep \nu_\vep$ and
it follows from the normalization $\int \la_\vep^{-1}\cL_\vep 1 \;d\nu_\vep=1$ that the convergence
$\lim_{n} \frac1n \log \|\la_\vep^{-n}\cL^n_\vep 1\|_r=0$
is uniform for all small $\vep$. Proceeding as in the proof of Proposition~\ref{propre} we deduce that the 
functions $\vep\to \la_\vep$ and $\vep\to d\mu_\vep/d\nu_\vep$ vary continuously for all small $\vep$. In fact, 
proceeding as before one obtains that
\begin{align*}
\left\|\frac{d\mu_{\vep}}{d\nu_{\vep}} -  \frac{d\mu_{f,\phi}}{d\nu_{f,\phi}} \right\|_{r} 
	& \leq 6 R_1 \Delta \tau^n 
	+ \left\|\la_\vep^{-n}\cL_{\vep}^{n}1- \la^{-n}_{f,\phi}\cL_{f,\phi}^{n}1  \right\|_{r}
\end{align*}
where $R_1$ is a uniform upper bound for the $C^0$-norm of the iterates $\la_\vep^{-n}\cL^n_\vep 1$
for all small $\vep$, the constant $\De$ is the diameter of the cone $\La_{\kappa}^r$ and $0<\tau<1$. 
Using that for $n$ fixed the functions $\cL^n_\vep 1$ and $\cL^n_{f,\phi} 1$ are uniformly close provided that
$\vep$ is small then one deduces that $\la_\vep \to \la_{f,\phi}$ and. consequently, that
the density $d\mu_{\vep}/d\nu_{\vep}$ converges to $d\mu_{f,\phi}/d\nu_{f,\phi}$ as $\vep\to0$.
This concludes the proof of our theorem.
\end{proof}

\section{Examples}\label{s.examples}

In this section we provide some examples and commment on our assumptions.

\begin{example}\label{ex.saddle}
Let $f_0:\torus^d\to\torus^d$ be a linear expanding map. Fix some
covering $\cU$ by domains of injectivity for $f_0$ and some $U_0\in\cU$ containing a fixed (or
periodic) point $p$. Then deform $f_0$ on a small neighborhood of
$p$ inside $U_0$ by a pitchfork bifurcation in such a way that $p$
becomes a saddle for the perturbed local diffeomorphism $f$. In particular,
such perturbation can be done in the $C^r$-topology, for every $r>0$. 
By construction, $f$ coincides with $f_0$ in the complement of $P_1$,
where uniform expansion holds. Observe that we may take the
deformation in such a way that $f$ is never too contracting in
$P_1$, which guarantees that conditions (H1) and (H2) hold, and that $f$ is still
topologically exact. Condition (P') is clearly satisfied by
any $C^r$-potential close to $\phi\equiv 0$. Hence, there exists a unique measure of 
maximal entropy $\mu$ for $f$, it is absolutely continuous with respect to a conformal measure $\nu$,
 supported in the whole manifold $\torus^d$ and  has exponential decay of correlations on the space $C^r$ observables. 
Moreover, it follows from our results that the density $d\mu/d\nu$ is $C^r$ 
and it varies continuously in the $C^{[r]}$-topology with the dynamical system $f$, 
where $[r]$ denotes the integer part of $r$.
Furthermore, the topological pressure function $\Ptop(f,\phi)$ varies continuously among the pairs
$(f,\phi)$ that satisfy conditions (H1), (H2) and (P') with uniform constants.
Finally, in the case that $r\ge 1$ we have that the maximal entropy measure is strong stable 
under deterministic perturbations and satisfies a random spectral stability.
\end{example}

In fact, the previous example can be modified to deal with expanding maps with indifferent 
periodic points in a higher-dimensional setting. A particularly interesting one-dimenional example is 
given by the Maneville-Pomeau transformation and the family of potentials $\varphi_t=-t\log|Df|$. 
An intermitency phenomenon occurs at $t=1$ but no longer occurs whenever $t$ is close to zero as we 
now discuss with detail.

\begin{example}\emph{(Manneville-Pomeau map)}
If $\al \in (0,1)$, let $f_\al:[0,1]\to [0,1]$ be the $C^{1+\al}$-local
diffeomorphism given by
\begin{equation*}\label{eq. Manneville-Pomeau}
f_\al(x)= \left\{
\begin{array}{cl}
x(1+2^{\alpha} x^{\alpha}) & \mbox{if}\; 0 \leq x \leq \frac{1}{2}  \\
2x-1 & \mbox{if}\; \frac{1}{2} < x \leq 1.
\end{array}
\right.
\end{equation*}
Observe that conditions (H1) and (H2) are verified and the family $\varphi_{\al,t}=-t\log|Df_\al|$ of
$C^\alpha$-potentials do satisfy condition (P) for all $|t|\le t_0$ small and $\al\in(0,1)$ since
\begin{equation*}
|\varphi_{\al,t}(x)-\varphi_{\al,t}(y)|
	= |t\log|Df_\al(x)|-t\log|Df_\al(y)||
	=|t| \log \frac{|Df_\al(x)|}{|Df_\al(y)|}
	\leq |t| \log ( 2+\al )
\end{equation*}
Hence, we obtain that for all $|t|\le t_0$ there exists a unique equilibrium state $\mu_t$, it is absolutely continuous
with respect to a conformal measure $\nu_t$ and has exponential decay of correlations in the space of H\"older 
observables. Moreover, $d\mu_t/d\nu_t$ is H\"older continuous and it varies continuously in the $C^0$-norm for 
all $|t|\le t_0$.

Moreover no transition occurs once one considers the order of contact $\al$ of the indifferent fixed point to
increase. Indeed, if $\al$ is arbitrary large then it follows from our previous reasoning that there exists a small 
interval $I_\al=[-t_\al,t_\al]$ containing zero such that the topological pressure $\mathbb R^+\times [-t_\al,t_\al] \ni (\al,t) \mapsto \Ptop(f_\al,\varphi_{\al,t})$ varies continuously. Moreover, there is a unique equilibrium state for
the $C^{\al}$-potential $\varphi_{\al,t}$ with $|t|\le t_\al$ and it is $C^{[\al]}$-strong stable under deterministic
perturbations: for every $(\al,t)\in\mathbb R^+\times[-t_\al,t_\al]$ there exists a unique equilibrium state $\mu_{\al,t}$ absolutely continuous with respect to a conformal measure $\nu_{\al,t}$, its density $d\mu_{\al,t}/d\nu_{\al,t}$ is $C^{[\al]}$ and varies continuously with $(\al,t)$. 
Finally, since our strong random spectral stability result applies for general random perturbations one can
consider e.g. $\theta_\vep$ to be the uniform distribution in the one-parameter family of pairs $\{(f_\al,\varphi_{\al,t_\al}): \al\in (\al_0-\vep,\al_0+\vep)\} \subset \cF\times \cW$. In particular, the random dynamical system associated
considers random orbits using maps with indifferent fixed points with different contact orders. 
Here our results yield that the random Ruelle-Perron-Frobenius operator $\cL_\vep$ has the spectral gap property
and that its spectral radius $\la_\vep$ converges to 
$\exp(\Ptop(f_\al,\varphi_{\al,t_\al}))\mid_{\al=\al_0}$ as $\vep$ tends to zero.
\end{example}

\bigskip

\textbf{Acknowledgements.} The authors are grateful to A. Arbieto, T. Bomfim, C. Matheus, K. Oliveira, V. Pinheiro 
and M. Viana for very fruitful conversations on thermodynamical formalism and to the anonimous 
referees for the careful reading of the manuscript and suggestions. This work was partially supported by CNPq-Brazil and FAPESB.

\bibliographystyle{alpha}

\end{document}